\documentclass[11pt,a4paper]{amsart}
\usepackage[draft,marginclue,nomargin,footnote]{fixme}
\usepackage{color}
\usepackage{enumitem} 
\newcommand\sbullet{\mbox{\tiny @}}
\def\ov{\overline}
\def\wt{\widetilde}
\def\DX{C_X}

\def\ot{\cdot}
\def\ba{\bar{a}}

\def\Set{{\mathsf{Set}}}
\def\Par{{\mathsf{Par}}}
\def\Pos{{\mathsf{Pos}}}
\def\Top{{\mathsf{Top}}}

\def\Vec{{\mathsf{Vec}}}
\def\JSL{{\mathsf{JSL}}}
\def\Par{{\mathsf{Par}}}
\def\Gra{{\mathsf{Gra}}}
\def\op{\mathrm{op}}

\def\Str{\mathsf{Str}}
\def\Un{\mathsf{Un}}
\def\pr{\prime}
\input xy
\xyoption{all} \xyoption{arc}
\usepackage{graphicx}
\usepackage{setspace}
\usepackage{kpfonts}

\usepackage{amsbsy,mathrsfs,latexsym,amssymb,mathbbol,stmaryrd}
\usepackage[mathcal]{euscript}

\theoremstyle{plain}
\newtheorem{theorem}{Theorem}[section]
\newtheorem{proposition}[theorem]{Proposition}
\newtheorem{corollary}[theorem]{Corollary}
\newtheorem{lemma}[theorem]{Lemma}

\theoremstyle{definition}
\newtheorem{definition}[theorem]{Definition}

\newtheorem{example}[theorem]{Example}
\newtheorem{examples}[theorem]{Examples}

\newtheorem{remark}[theorem]{Remark}

\newtheorem{notation}[theorem]{Notation}
\newtheorem{observation}[theorem]{Observation}

\newtheorem{noname}[theorem]{}

\usepackage{cite}
\def\2{\textbf{2}}
\def\aa{\ast\ast}

\def\id{\mathrm{id}}

\def\Id{{\mbox{Id}}}

\def\Set{{\mathsf{Set}}}
\def\Pos{{\mathsf{Pos}}}
\def\Top{{\mathsf{Top}}}

\def\N{{\rm I\kern-.20em N}}
\def\R{{\rm I\kern-.17em R}}
\def\Z{{\rm Z\kern-.32em Z}}

\newcommand{\gl}{\lambda}

\newcommand\A{\mathcal {A}}
\newcommand\ca{\mathcal {A}}

\newcommand\X{\mathcal {X}}

\newcommand\cf{\mathcal{F}}
\newcommand\cu{\mathcal{U}}

\newcommand\cg{\mathcal{G}}
\newcommand\ch{\mathcal{H}}

\newcommand\cp{\mathcal {P}}

\newcommand\ck{\mathcal {K}}

\newcommand{\TT}{\mathbb{T}}

\newcommand{\sfp}{f\hspace*{-0.5mm}p}

\begin{document}

\title[D-Ultrafilters and their Monads]
{D-Ultrafilters and their Monads}

\author{Ji\v r\' {\i} Ad\'amek}
\address{Department of Mathematics, Faculty of Electrical Engineering\\ Czech Technical University in Prague, Czech Republic
\& Institute of Theoretical Computer Science, Technical University of Braunschweig, Germany}
\author{Lurdes Sousa}
\address{ University of Coimbra, CMUC, Department of Mathematics, Portugal  \& Polytechnic Institute of Viseu, ESTGV, Portugal}
\thanks{J.~Ad\'amek was supported by the Grant Agency of the Czech Republic under the grant 19-00902S. \\ 
L.~Sousa was partially supported   by the Centre for Mathematics of the University of Coimbra - UIDB/00324/2020, funded by the Portuguese Government through FCT/MCTES}


\begin{abstract} 
 For a number of  locally finitely presentable categories $\ck$ we describe the codensity monad of the full embedding of all finitely presentable objects into $\ck$. We introduce the concept of $D$-ultrafilter on an object, where $D$ is a ``nice'' cogenerator of $\ck$. We prove that the codensity monad assigns to every object an object representing all $D$-ultrafilters on it. Our result covers e.g. categories of sets, vector spaces, posets, semilattices, graphs and $M$-sets for finite commutative monoids $M$. 
\end{abstract}

\maketitle

\section{Introduction}

We present a generalization of the concept of an ultrafilter on a set: for a number of categories $\ck$ we define $D$-ultrafilters on an object of $\ck$. Here $D$ is a cogenerator of $\ck$ with a special property; we speak about $\ast$-cogenerators, see 
below. For example $D=\{0,1\}$ is a $\ast$-cogenerator of $\Set$, in this case  $D$-ultrafilters are the usual ultrafilters. By a classical result of Kennison and Gildenhuys \cite{KG} the ultrafilter monad on $\Set$ (assigning to every set the set of all ultrafilters) is the codensity monad of the embedding $\Set_{\sfp}\hookrightarrow \Set$ of finite sets. We will prove that, in general, the corresponding monad of $D$-ultrafilters on $\ck$ is the codensity monad of the embedding $\ck_{\sfp}\hookrightarrow \ck$ of finitely presentable objects of $\ck$.

 	We consider closed monoidal  categories.  Our examples include all commutative  varieties, for instance, vector spaces,  semilattices or $M$-sets for finite commutative monoids $M$. Recall that a variety of algebras is closed monoidal  with respect to the usual tensor product if and only if  it is commutative (aka entropic), see \cite{BN}.
 Another sort of examples are cartesian closed categories such as posets or graphs. 

All of our examples (except the last section presenting some generalizations) are locally finitely presentable categories in the sense of Gabriel and Ulmer \cite{GU}. One of the most important features of locally finitely presentable  categories  $\ck$ is that the full embedding   
$$E:\ck_{\sfp}\hookrightarrow \ck$$ 
of all finitely presentable  objects  is dense, i.e.  every object   $X$ is a canonical colimit of all morphisms  $a:A\to X$ with $A$ finitely presentable. More precisely: the forgetful functor $\ck_{\sfp}\,/ X\to \ck$ of the coslice category  has colimit $X$ with the canonical colimit cocone.

Not surprisingly, finitely presentable  objects  are usually \textit{not} codense. A measure of how ``far away'' a functor $E$ is from being codense is the {codensity monad }$\TT$ of $E$. This monad is given  by the right Kan extension of $E$ along itself:
$$T=\mbox{Ran}_EE,$$
see below. For codense functors $E$, this is the trivial monad $\mbox{Id}$.

 Recently, Leinster proved that the codensity monad  of the embedding of finite-dimensional vector  spaces into the category  $K$-$\Vec$ of vector spaces over a field $K$ is the \textit{double-dualization monad  }
$$TX=X^{\aa}.$$
And he asked for a general description of  the codensity monad of $E:\ck_{\sfp}\hookrightarrow \ck$  for locally finitely presentable categories  $\ck$.

The purpose of our paper is to answer to Leinster's question. Not for general locally finitely presentable  categories,   but for quite some. Given a cogenerator $D$ we denote by $(-)^{\ast}=[-,D]$  the contravariant endofunctor $X\mapsto [X,D]$; then $D$ is a \textit{$\ast$-cogenerator} if for every object $X$ the dual object $X^{\ast}$ is a canonical colimit of objects  $A^{\ast}$ with $A$ finitely presentable. We prove that in all our examples the given cogenerator $D$  is a $\ast$-cogenerator.  The composite $(-)^{\aa}$ of $(-)^{\ast}$ with itself is the well-known double-dualization monad (relative to $D$).

We introduce the concept of a $D$-ultrafilter on an object   $X$ and form the corresponding $D$-ultrafilter monad on $\ck$ as a submonad of the double-dualization monad. This turns out to be the desired  codensity monad  of $E: \ck_{\sfp}\hookrightarrow \ck$. Example: in the category  of posets the 2-chain is a $\ast$-cogenerator. Here $X^{\ast}$ is the poset of all $\uparrow$-sets of $X$, ordered by inclusion. Therefore $X^{\aa}$ is the poset of all upwards closed collections $\mathcal{W}$ of $\uparrow$-sets, again ordered by inclusion. A $D$-ultrafilter on $X$ is such a nonempty collection $\mathcal{W}$ which is

\vskip1mm

(i) closed under finite intersections,

\noindent and

(ii) \textit{prime}, i.e., it does not contain  $\emptyset$ and if it contains $R\cup S$, then it contains $R$ or $S$.

\vskip1mm

\noindent This is analogous to the classical ultrafilters on sets, which are nonempty, upwards closed, prime collections of subsets,  closed under finite intersections.

Analogously in all examples that our result covers: the codensity monad  $\mathbb{T}$ assigns to every object   $X$ an object   formed by all $D$-ultrafilters on $X$,   and there is a close analogy between the latter and the classical ultrafilters. Moreover, we prove that  $\TT$ is also the enriched codensity monad of the embedding $E$.

\vspace{3mm}

\textbf{On codensity monads.} Recall that for every functor $E:\ca\to \ck$ the codensity monad  is defined as the right Kan extension along itself,  $T=\mathrm{Ran}_E E$.   That is, $T$ is an endofunctor endowed with a natural transformation $\tau: TE\to E$ universal among natural transformations from $(-)\cdot E$ to $E$. Applying the universal property to $\id: \Id \cdot E\to E$ we get a unique natural transformation $\eta:\Id\to T$. And applying it to $\tau \cdot T\tau: TTE\to E$ we get a unique natural transformation $\mu:TT\to T$. Then $(T,\eta, \mu)$ is a monad, see \cite{L}.

If $\ca$ (like $\ck_{\sfp}$ above) is an essentially small full subcategory of a complete category $\ck$, then the codensity monad   of the embedding $E:\ca\to \ck$ is obtained by the following \textit{limit formula}: for every object   $X$ denote by 
$$\DX:X/\ca\to \ck$$
 the functor 
assigning to every arrow $a:X\to A$ the codomain $A$, and put
$$TX=\lim \DX.$$
We have a limit cone denoted by $\, \psi_a:TX\to A\,$ for $(A, a)\in X/\ca$. On  morphisms $f:X\to Y$,   $Tf$  is defined as follows: there exists a unique morphism $Tf:TX\to TY$ with 
$$\psi_a\ot Tf=\psi_{a\ot f}\qquad \text{for all $a:Y\to A$ in $Y/\ca$.}$$
The unit
 $\eta_X^T: X\to TX$  is the unique morphism given by 
 $$\psi_a\ot \eta_X^T=a\qquad \text{for all $a:X\to A$ in $X/\ca$}$$
 and the multiplication is defined by the following commutative triangles
\begin{equation*}
\qquad \qquad \xymatrix{TTX\ar[rr]^{\mu^T_X}\ar[rd]_{\psi_{\psi_a}}&&TX\ar[ld]^{\psi_a}\\
&A&}   \qquad \text{for all $a:X\to A$ in $X/\ca$.}
\end{equation*}

\vspace{3mm}

\textbf{Related work.} As mentioned already, our paper was inspired by that of Leinster \cite{L}. A related topic was discussed in the PhD thesis of Barry-Patrick Devlin \cite{D}. He also aimed to describe codensity monads of embeddings of ``finite-objects'', and  he also introduced a concept of ultrafilter on an object.  However his thesis is fundamentally  disjoint from our paper. For example, the categories  he works with are varieties whose monads contain that of abelian groups as a submonad -- the only example on our list above with this property is $K$-$\Vec$, see  Example \ref{E3.6} below.

\vspace{3mm}

\textbf{Acknowledgement.}  We are grateful to the referee for a number of excellent suggestions, in particular, for the proof of Proposition \ref{P2.6} unifying our (previously disparate) examples in \ref{E-star}.  We also thank Ji\v r\'{\i} Velebil for helpful discussions on  enriched codensity monads.

\section{$\ast$-cogenerators}

Throughout  we work with a symmetric monoidal closed category  $(\ck, \otimes, I)$ with a specified object   $D$.

The functor $[-,D]:\ck \to \ck^{\op}$ is denoted by $(-)^{\ast}$. Since it is left adjoint to its dual, we obtain a monad $(-)^{\aa}$ on $\ck$ given by
 $$X^{\aa}=[[X,D],D]$$
  called the \textit{ double-dualization monad}. Its unit 
  $$\eta_X:X\to [[X,D], D]$$
  is the transpose of the evaluation map $[X,D]\otimes X\to D$ precomposed with the symmetry isomorphism $X\otimes [X,D]\xrightarrow{\cong}[X,D]\otimes X$. Its multiplication is given by $\mu_X=\eta^{\ast}_{X^{\ast}}:X^{\aa\aa}\to X^{\aa}$.
  
  \begin{remark}\label{newR2.1} $(-)^{\ast}$ can be described on  morphisms $f:X\to Y$ as the unique morphism $f^{\ast}:Y^{\ast}\to X^{\ast}$ for which the square below commutes:
  $$\xymatrix{%
  X\otimes Y^{\ast}\ar[d]_{f\otimes Y^{\ast}}\ar[r]^{X\otimes f^{\ast}}&X\otimes X^{\ast}\ar[d]^{\text{ev}}\\
  Y\otimes Y^{\ast}\ar[r]_{\text{ev}}&D}%
  $$
  \end{remark}
  
  \begin{examples}\label{E2.1}Most of our examples are \textit{commutative varieties}  of finitary algebras. Recall that a variety $\ck$ is called commutative (or entropic) if for each of its $n$-ary operation symbols $\sigma$ and every algebra $K\in \ck$ we have a homomorphism $\sigma_K:K ^n \to K$. Let $|-|$ denote the forgetful functor  into $\Set$. Every variety is symmetric monoidal w.r.t. the usual tensor product:
  $$\mbox{$A\otimes B$ represents bimorphisms from $|A|\times |B|$}$$
  and the unit
$$\mbox{$I$ is the free algebra on one generator.}$$
As proved by Banaschewski and Nelson \cite{BN}, $\ck$ is a monoidal closed category iff it is a commutative variety. Then, for arbitrary objects  $A$ and $B$, all morphisms  in $\ck(A,B)$ form a subalgebra of the power $B^{|A|}$ which yields the object  $[A,B]$. Another equivalent formulation, as observed by Linton \cite{Li}, is that the monad on $\Set$ associated with $\ck$ is commutative in Kock's sense \cite{K}.

Here are our leading examples of commutative varieties with a specified cogenerator $D$. Observe that in each case all finite powers $D^n$ of $D$ are finitely presentable algebras.
\begin{enumerate}[label=(\alph*)]
 \item $\Set$ with $D=\{0,1\}$.
 Here $(-)^{\ast}$ is the contravariant power-set functor $\cp$, thus $X^{\aa}= \cp\cp X$ consists of all collections of subsets of $X$. For a function $f:X\to Y$ the function $f^{\aa}$ takes a collection $\cu\subseteq \cp X$ to 
 $$f^{\aa}(\cu)=\{R\subseteq Y\mid f^{-1}(R)\in \cu\}.$$ 
 And  $\eta_X$ assigns every element $x$ of $X$ the trivial ultrafilter $\eta_X(x)=\{R\subseteq X \mid x\in R\}$.
 \item\label{par}
 $\Par$, the category of sets and partial functions, with $D=\{1\}$. This is completely analogous, $X^{\aa}= \cp\cp X$.
 \item\label{kvec}
 $K$-$\Vec$, the category  of vector spaces over a field $K$ chosen as the cogenerator $D$. This example was the motivation for our notation: $X^{\ast}$  is the usual dual space (of all linear forms on $X$). Thus $X^{\aa}$ is the double-dual.
 For a linear function $f:X\to Y$ , the function $f^{\aa}$ assigns to every $a:X^{\ast}\to K$ in $X^{\aa}$ the element $a\ot f^{\ast}:Y^{\ast}\to K$ of $Y^{\aa}$. And $\eta_X: X\to X^{\aa}$ assigns to $x\in X$ the evaluation-at-$x$ of linear forms.
\item\label{jsl}
$\JSL$, the category  of join-semilattices (i.e., posets with finite joins) and homomorphisms, with $D=\2$, the chain $0<1$. Observe that homomorphisms preserve 0, the join of $\emptyset$.
Given a semilattice $X$, every  homomorphism $f:X\to \2$ defines a subset of $X$ by $f^{-1}(1)$. This is an $\uparrow$-set which is \textit{ prime}, i.e.,  it does not contain  $0$ and whenever it contains $x_1\vee x_2$, then it contains $x_1$ or $x_2$. Conversely, every prime $\uparrow$-set $R$ of $X$ defines a homomorphism $f_R:X\to \2$ by $f_R(x)=1$ iff $x\in R$. We can thus identify 
$$X^{\ast}= \text{ all prime $\uparrow$-sets of $X$}$$
ordered by inclusion. (The least element of $X^{\ast}$ is $\emptyset$.) Consequently,
$$X^{\aa}=\text{ all prime upwards closed collections of prime $\uparrow$-sets of $X$.}$$
Here a collection is called prime if it does not contain the empty set and whenever it contains $R_1\cup R_2$, then it contains $R_1$ or $R_2$. $X^{\aa}$ is also ordered by inclusion. Its smallest element is the empty collection.

  \item\label{mset} $M$-$\Set$, the category  of sets with an action of a monoid $M$. We assume that $M$ is commutative (so that $M$-$\Set$ is a commutative variety) and finite.
  We need the latter assumption to have a finitely presentable cogenerator. Recall that an $M$-set is a set $X$ equipped with a function from $M\times X$ to $X$ (notation: $(m,x)\mapsto mx$) such that the corresponding map from $M$ to $\Set(X,X)$ is a monoid homomorphism. Homomorphisms $f:X\to Y$ of $M$-sets, called \textit{equivariant maps}, are functions satisfying $f(mx)=mf(x)$.  A cogenerator of $M$-$\Set$ is 
  the power-set
  $$D=\cp M,$$
with the monoid action
$$mR=\{x\in  M\, |\,  mx \in R\} \quad \text{for $R\subseteq M,\; m\in M$.}$$
To see that this is indeed a cogenerator, observe that  equivariant maps  $g:X\to \cp M$  correspond bijectively to subsets (not just subalgebras!) of $X$: to every subset $Y\subseteq X$ assign $g_Y:X\to \cp M$ defined by 
$$g_Y(x)= \{ m\in M \mid mx\in Y\}\; \, \text{for all $x\in X$.}$$
The inverse assignment takes  every 
$g:X\to \cp M$ to $Y=\{mx\mid  x\in X,\, m\in g(x)\}$.


Thus for every $M$-set $X$ we conclude that
$$X^{\ast}=\cp X$$
is the power-set of the (underlying set of) $X$ with the monoid action $mY=\{x\in X\mid mx\in Y\}$. And the monoid action of
$X^{\aa}=\cp\cp X$
assigns  to $\cu\subseteq \cp X$ and $m\in M$ the result $m\cu=\{Y\subseteq X\mid mY\in \cu\}$.
\end{enumerate}
\end{examples}

\begin{examples}\label{E2.2}
Further we consider some cartesian closed categories with a cogenerator $D$. 
\begin{enumerate}[label=(\alph*)]
\item\label{pos}
 $\Pos$, the category of posets and monotone maps, with $D=\2$, the chain $0<1$. Here $[A,B]=\Pos(A,B)$ ordered pointwise. Thus, analogously to $\JSL$ above, 
 $$X^{\ast}=\text{ all $\uparrow$-sets of $X$}$$
 (ordered by inclusion) and
 $$X^{\aa}=\text{ all upwards closed collections of $\uparrow$-sets,}$$
 also ordered by inclusion.
\item\label{gra}
$\Gra$, the category  of undirected graphs and homomorphisms. Thus an object $A$ consists of a set $V_A$ of vertices and a symmetric relation $E_A\subseteq V_A\times V_A$ of edges. In case $E_A=V_A\times V_A$ we speak about the \textit{complete graph} on $V_A$. $\Gra$ has a cogenerator $D$, the complete graph on $\{0,1\}$.  Given graphs $A$ and $B$, then the hom-object
$$[A,B]=\Set(V_A,V_B)$$
consists of all functions, not only homomorphisms, and its edges are defined pointwise: they are all pairs of functions $(f,g)$ with 
\begin{equation}\label{gra1}(a,a')\in E_A\Rightarrow (f(a),g(a'))\in E_B, \text{ for all $a,a'\in V_A$}.\end{equation}
 Observe that loops of $[A,B]$ are precisely the homomorphisms from $A$ to $B$:
\begin{equation}\label{gra2}(f,f)\in E_{[A,B]}  \text{ iff  $f:A\to B$ is in $\Gra$}.\end{equation}
We conclude that
 $$X^{\ast}=\text{ complete graph on $\cp V_X$}$$
and
 $$X^{\aa}=\text{ complete graph on $\cp\cp V_X$.}$$


 \item
 $\Sigma$-$\Str$, the category of relational structures, where $\Sigma$ is a  signature of finitely many finitary symbols. (We allow only finitely many symbols to make sure that the terminal object is finitely presentable.) Objects $X$, $\Sigma$-structures, consist of a set $V_X$ and an $n$-ary relation $\sigma_X\subseteq V_X^n$ for every $\sigma \in \Sigma$ $n$-ary. Analogously to \ref{gra} we choose as $D$ the complete structure on $\{0,1\}$, that is, $\sigma_D=\{0,1\}^n$ for every $n$-ary symbol $\sigma$. Then
 $$X^{\aa}=\text{ complete $\Sigma$-structure on $\cp \cp V_X$.}$$
\end{enumerate}
\end{examples}

\begin{remark}\label{R2.3}
(1) Recall that a full subcategory $\ch$  of $\ck$ is  {\em dense}
 if the functor
$$E_{\ch}:\ck\to [\ch^{\op}, \Set], \;\;\; K\mapsto \big(\ck(-,K):\ch^{\op}\to \Set\big),$$
is fully faithful. In other words, every object $K$ is a canonical colimit of the diagram $\ch/K\to \ck$ given by $(H\xrightarrow{h} K)\mapsto H$.

(2) As explained in the introduction, we want to describe the codensity monad of the full embedding
$$\ck_{\sfp}\hookrightarrow \ck$$
of the subcategory of finitely presentable objects.  Recall that in case $\ck$ is locally finitely presentable, $\ck_{\sfp}$ is dense.
We are, however,  not assuming that $\ck$ is locally finitely presentable. Instead, we need that every object  of the form  $X^{\ast}$ is a canonical colimit of all $A^{\ast}$ with $A$ finitely presentable.
\end{remark}

\begin{notation}\label{N2.4}Recall that, for every object   $X$, the diagram  $C_X:X/\ck_{\sfp}\to \ck$ assigns to $a:X\to A$ with $A\in \ck_{\sfp}$ the codomain.  We denote the composite $(-)^{\ast}\cdot\big( C_X\big)^{\op}$ by $C_X^{\ast}$. That is, 
$$\qquad \qquad C^{\ast}_X:(X/\ck_{\sfp})^{\op}\to \ck\quad \text{with }\; \;  C_X^{\ast}(A,a)=A^{\ast}.$$
\end{notation}

\begin{definition}\label{D2.5}
An object   $D$ is called a \textit{$\ast$-object} provided that for all objects  $X$ we have $X^{\ast}=\mbox{colim} C^{\ast}_X$ with the canonical colimit cocone $a^{\ast}:A^{\ast}\to X^{\ast}$. If $D$ is a cogenerator, we speak about \textit{$\ast$-cogenerator}.
\end{definition}

The next proposition implies, as we demonstrate below, that in all the above examples $D$ is a $\ast$-cogenerator.

\begin{proposition}\label{P2.6} An object $D$  is  a $\ast$-object  whenever $\ck$ has a full, dense subcategory $\ch$ whose objects have finitely presentable duals w.r.t. $D$.
\end{proposition}

\begin{proof}
Let $X$ be an arbitrary object   and suppose that a cocone of $C^{\ast}_X$ with codomain $Z$ is given as follows
$$\begin{tabular}{cc}
$X\xrightarrow{a}A$\\
\hline \\ [-2.5ex]
$A^{\ast}\xrightarrow{\ba}Z$
\end{tabular}
\qquad \text{ for $(a,A)\in X/\ck_{\sfp}$.}$$
We are to prove that there exists a unique morphism  $f$  making the following triangles 
$$\xymatrix{&A^{\ast}\ar[ld]_{a^{\ast}}\ar[rd]^{\ba}&\\X^{\ast}\ar[rr]_f&&Z} \qquad \begin{array}{c}\\(a,A)\in X/\ck_{\sfp}\end{array}$$
commutative. 

(1) Uniqueness. We prove that the cocone of all $a^{\ast}$ is collectively epic. Since $E_{\ch}$ is fully faithful (see Remark \ref{R2.3}(1)), it is sufficient to prove this about   the cocone of all $E_{\ch}a^{\ast}$. That is, for every object $H$ in $\ch$ and every morphism $h:H\to X^{\ast}$ there exists $(A,a)$ in $X/\ck_{\sfp} $ such that $h=a^{\ast}\cdot k$ for some $k: H\to A^{\ast}$. In fact, put $A=H^{\ast}$ (lying in $\ck_{\sfp}$) and $a=h^{\sbullet}$  where we denote by 
$$h^{\sbullet}=h^{\ast}\cdot
\eta_X:X\to H^{\ast}$$
the adjoint transpose. Since  it fulfils $(h^{\sbullet})^{\ast}\cdot \eta_H=h$, for $k=\eta_H$ we get the desired equality.

(2) Existence. Since  $\ch$ is dense in $\ck$, $X^{\ast}$ is the canonical colimit of the  diagram $\ch/X^{\ast}\to \ck$ taking every $H\xrightarrow{h}X^{\ast}$ to $H$. Applying $(-)^{\ast}$ to the   colimit cocone and composing with $\eta_X$, we obtain the cone
$$X\xrightarrow{h^{\sbullet}}H^{\ast},\; h\in \ch/X^{\ast}$$
which is a subcone of $C_X:X/\ck_{\sfp}\to \ck$,    since all $H^{\ast}$ lie in $\ck_{\sfp}$. For each $h$ we form $\ov{h^{\sbullet}}:H^{\aa}\to Z$.
 The
following assignment defines a cocone of the diagram $\ch/X^{\ast}\to \ck$:
$$\begin{array}{c}H\xrightarrow{h}X^{\ast}\\
\hline
\vspace*{-4mm}
\\
H\xrightarrow{\eta_H}H^{\aa}\xrightarrow{\overline{h^{\sbullet}}}Z \\
\end{array}$$
Indeed, for every connecting morphism
$$\xymatrix{H\ar[rr]^{u}\ar[dr]_{h}&&K\ar[dl]^{k}\\
&X^{\ast}&} \qquad (H,K\in \ch)$$
we have a morphism $u^{\ast}:k^{\sbullet}\to h^{\sbullet}$ in $X/\ck_{\sfp}$, therefore the triangle below commutes:
$$\xymatrix@=1.2em{H^{\aa}\ar[ddrrr]_{\overline{h^{\sbullet}}}\ar[rrrrrr]^{u^{\aa}}&&&&&&K^{\aa}\ar[llldd]^{\overline{k^{\sbullet}}}\\
&&H\ar@{.>}[llu]_{\eta_H}\ar@{.>}[rr]_u&&K\ar@{.>}[rru]^{\eta_K}&&\\
&&&Z&&&}$$
This yields the desired equality 
$$\overline{h^{\sbullet}}\cdot \eta_H=\left(\overline{k^{\sbullet}}\cdot \eta_K\right)\cdot u$$
since $\eta$ is natural.

We thus have a unique morphism
$$f:X^{\ast}\to Z \; \; \text{ with }\;  f\cdot h=\overline{h^{\sbullet}}\cdot \eta_H$$
for all $(H,h)$ in $\ch/X^{\ast}$. We are going to prove that this implies $f\cdot a^{\ast}=\overline{a}$ for all $a:X\to A$, $A\in \ck_{\sfp}$. Since $\ch$ is dense, it is sufficient to prove, for all morphisms $k:H\to A^{\ast}$ with $H\in \ch$, that
$$f\cdot a^{\ast}\cdot k=\overline{a}\cdot k.$$
Apply  $f\cdot h=\overline{h^{\sbullet}}\cdot \eta_H$ to $h=a^{\ast}\cdot k:H\to X^{\ast}$. It is clear that $h^{\sbullet}=k^{\sbullet}\cdot a$, thus, we get a morphism
of $X/\ck_{\sfp}$. Therefore
$$\overline{h^{\sbullet}}=\overline{a}\cdot \big(k^{\sbullet}\big)^{\ast}.$$
This yields the desired result:
$$\overline{a}\cdot k=\overline{a}\cdot \big(k^{\sbullet}\big)^{\ast}\cdot \eta_H=\overline{h^{\sbullet}}\cdot \eta_H=f\cdot h=f\cdot a^{\ast}\cdot k.$$
\end{proof}

\begin{examples}\label{E-star} (a) In a commutative variety every object $D$ with all $D^n$ finitely presentable ($n\in \mathbb{N}$) is a $\ast$-object. Indeed, let $\ch$ be the full dense subcategory of all free algebras on $n$ generators for $n\in \mathbb{N}$.

(b) Every finite poset is a $\ast$-object of $\Pos$. Here we use the density of the 2-element chain.

(c) Every finite graph is a  $\ast$-object of $\Gra$. We can take as $\ch$ the dense subcategory consisting of the singleton discrete graph and the single-edge graph.

(d) Every finite $\Sigma$-structure is a  $\ast$-object of $\Sigma$-$\Str$. Here we use $\ch$ consisting of (a) a discrete one-element structure and (b) for each symbol $\sigma \in \Sigma$ of arity $n$  the $n$-element structure $A$ on $\{0,...n-1\}$ with all relations but $\sigma_A$ empty, and $\sigma_A$ containing just one n-tuple $(0,...,n-1)$.

\end{examples}

 Observe that in all examples of \ref{E2.1} and \ref{E2.2} the unit object $I$ is finitely presentable. For commutative varieties, where $I$ is the free algebra on one generator, this is automatic. In the cartesian closed categories $\Pos$ and $\Gra$ this also holds. For $\Sigma$-$\Str$ the terminal object $I=1$ is finitely presentable since we assume that $\Sigma$ is finite.

\vspace*{0.5mm}
\begin{proposition}\label{P2.new} All $\ast$-objects are finitely presentable, assuming that $I$ is.
\end{proposition}

\begin{proof}
Denote by $\rho_A:A\otimes I\to A$ the right-unit isomorphism. 
If $D$ is a $\ast$-object, then $D^{\ast}$ itself is a filtered colimit of the diagram $(D/\ck_{\sfp})^{\op}\to \ck$ with the colimit cocone $a^{\ast}: A^{\ast}\to D^{\ast}$. The transpose $\widehat{\rho_D}:I\to [D,D]$ of $\rho_D:D\otimes I \to D$ factorizes, since $I$ is finitely presentable, through one of the colimit maps $a^{\ast}$. The factorizing morphism from $I$ to $[A,D]$ is a transpose $\widehat{u\ot \rho_A}:A\otimes I\to D$ for a morphism  $u:A\to D$:
$$\xymatrix{I\ar[rr]^{\widehat{\rho_D}}\ar[dr]_{\widehat{u\ot \rho_A}}&&[D,D]\\
&[A,D]\ar[ru]_ {a^{\ast}=[a,D]}&}$$
We obtain a commutative triangle by multiplying  the above one with $D$:
$$\xymatrix{D\otimes I\ar[rrrr]^{D\otimes\widehat{\rho_D}}
\ar@<-2ex>[dddrr]_{D\otimes \widehat{u\ot \rho_A}}\ar[drr]_{\rho_D}&&&&D\otimes [D,D]\ar[dll]^{\text{ev}}\\&&D&&\\&&A\otimes [A,D]\ar[u]^{\text{ev}}&&\\
&&\stackrel{~}{\qquad \qquad \qquad \qquad D\otimes [A,D]\qquad \qquad \qquad \qquad } \ar[u]^{a\otimes [A,D]}\ar@<-2ex>[rruuu]_ {D\otimes a^{\ast}}&&}$$
Moreover, the upper triangle commutes by the definition of transpose, and the right-hand one does by Remark \ref{newR2.1}. Consequently,  the left-hand triangle also commutes. Consider the following diagram, using the above triangle in its left-hand part:
$$
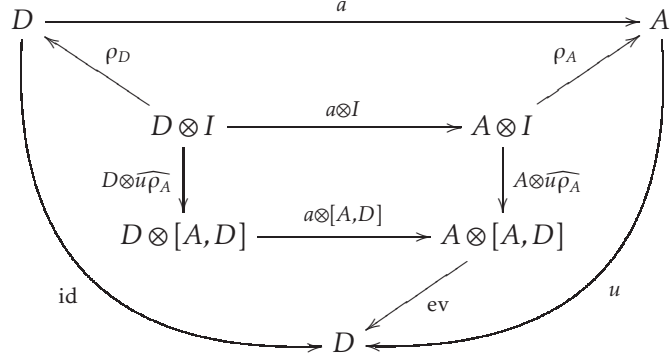
\xymatrix{D\ar[rrrr]^{a}\ar@/_50pt/[dddrr]_{\id}&&&&A\ar@/^50pt/[llddd]^u\\
&D\otimes I\ar[lu]_{\rho_D}\ar[rr]^{a\otimes I}\ar[d]_{D\otimes \widehat{u\rho_A}}&&A\otimes I\ar[d]^{A\otimes\widehat{u\rho_A}}\ar[ru]^{\rho_A}&\\
&D\otimes [A,D]\ar[rr]^{a\otimes [A,D]}&&A\otimes [A,D]\ar[dl]^{\text{ev}}&\\
&&D&&}$$
The right-hand part commutes by the definition of transpose, the upper part by naturality of $\rho$, and the middle square commutes since both passages yield $a\otimes \widehat{u\ot\rho_A}$. This proves $u\ot a=\id$. Thus $D$ is a split quotient of $A\in \ck_{\sfp}$, concluding the proof.
\end{proof}

\begin{example}\label{E-PN}  A commutative variety with a cogenerator does not have to  possess a $\ast$-cogenerator.  An example is the variety 
$$\Un$$
 of unary algebras on one operation. This is equivalent to $\mathbb{N}$-$\Set$, the category of sets with the action of   the additive monoid $\mathbb{N}$ of natural numbers. It has a cogenerator analogous to that of Example \ref{E2.1}\ref{mset}: take $\cp \mathbb{N}$ with the unary operation  sending $V\subseteq \mathbb{N}$ to $\{n-1 \mid n\in V, n\not=0\}$.

Assuming that $\Un$ has a $\ast$-cogenerator $D$, we derive a contradiction as follows:

 The operation of $D$ forms some cycles, and since $D$ is by Proposition \ref{P2.new} finitely generated, there exists a prime $n$ such that all cycles of $D$ have lengths smaller than $n$. But then $D$ is not a cogenerator: if $A$ is an algebra consisting of a cycle of length $n$, there exists no non-constant homomorphism from $A$  to $D$.
\end{example}

\begin{proposition}\label{P2.8} For every cogenerator $D$ the unit of the double-dualization monad is monic.
\end{proposition}

\begin{proof} (1) $\eta_D$ is monic. Indeed, by definition, $\eta_D$ is the transpose of the composite
$$D\otimes D^{\ast}\xrightarrow{s}D^{\ast}\otimes D\xrightarrow{\mathrm{ev}}D$$
where $s$ is the symmetry. Thus we have a commutative triangle as follows:
$$
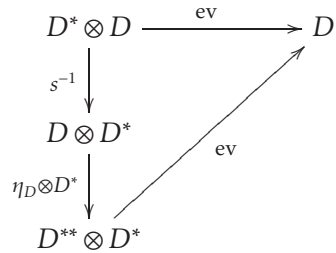
\xymatrix{D^{\ast}\otimes D\ar[rr]^{\mathrm{ev}}\ar[d]_{s^{-1}}&&D\\
D\otimes D^{\ast}\ar[d]_{\eta_D\otimes D^{\ast}}&&\\
D^{\aa}\otimes D^{\ast}\ar[uurr]_{\mathrm{ev}}}$$
Denote by $i:I\to D^{\ast}$ the transpose of the left-unit isomorphism $\lambda_D:I\otimes D\to D$:
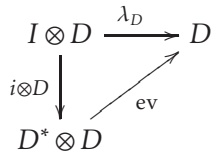
$$\xymatrix{I\otimes D\ar[r]^{\lambda_D}\ar[d]_{i\otimes D}&D\\
D^{\ast}\otimes D\ar[ur]_{\mathrm{ev}}}$$
Thus the following diagram commutes (due to naturality of $s$):
$$\xymatrix{I\otimes D\ar[rr]^{\lambda_D}\ar[d]_{i\otimes D}&&D\\
D^{\ast}\otimes D\ar[d]_{D^{\ast}\otimes \eta_D}&&\\
D^{\ast}\otimes D^{\aa}\ar[d]_{s^{-1}}&&\\
D^{\aa}\otimes D^{\ast}\ar[uuurr]_{\mathrm{ev}}}$$
Therefore, $i\otimes \eta_D$ is a split monomorphism (with splitting $\lambda_D^{-1}\ot \mathrm{ev}\ot s^{-1}$).

Consequently, given morphisms  $u_1, \, u_2: Y\to D$ with $\eta_D\ot u_1=\eta_D\ot u_2$, then $I\otimes u_1=I\otimes u_2$ (since $i\otimes \eta_D$ merges that last pair) which proves $u_1=u_2$, because $I\otimes - \cong \Id_{\ck}$.

(2) For every object   $X$ the morphism  $\eta_X$ is monic. Indeed, given $u_1,\, u_2: Y\to X$ with $u_1\not=u_2$, there exists, since $D$ is a cogenerator, a morphism  $f:X\to D$ with $f\ot u_1\not=f\ot u_2$.  Hence, by (1), $\eta_D\ot f\ot u_1\not= \eta_D \ot f\ot u_2$. The following commutative diagrams
$$\xymatrix{Y\ar[r]^{u_i}&X\ar[r]^{\eta_X}\ar[d]_{f}&X^{\aa}\ar[d]^{f^{\aa}}\\
&D\ar[r]_{\eta_D}&D^{\aa}} \qquad \quad  \begin{array}{l}\\ \\
(i=1,2)
\end{array}$$
prove $\eta_X\ot u_1\not=\eta_X\ot u_2$.
\end{proof}

\section{$D$-ultrafilters}

We assume in this section that a finitely presentable cogenerator $D$ in a symmetric monoidal category  $\ck$ with preimages is given. Recall from Proposition \ref{P2.8} that each $\eta_A:A\to A^{\aa}$ is monic. 

\begin{definition}\label{D:derived} (1) Given a morphism  $a:X\to A$ with $A$ finitely presentable, we call the preimage of $\eta_A$ under $a^{\aa}: X^{\aa}\to A^{\aa}$ the \textit{derived subobject} of $a$. We use the following notation for the corresponding pullback:
\begin{equation}\label{derived}\xymatrix{A'\; \ar[d]_{p(a)}\ar@{>->}[r]^{a'}&X^{\aa}\ar[d]^{a^{\aa}}\\
A\; \ar@{>->}[r]_{\eta_A}&A^{\aa}}\end{equation}
(2)
If $\ck$ is a concrete category over $\Set$,  a  \textit{ $D$-ultrafilter} on an object $X$ is an element of the underlying set of the intersection of all derived subobjects of $X^{\aa}$.
\end{definition}

\begin{example}\label{E3.2} In $\Set$ with $D=\{0,1\}$, this is precisely an ultrafilter on $X$. Recall that an ultrafilter is  a nonempty collection $\cu$ of  subsets that is upwards closed, closed under finite intersections and \textit{ prime} (i.e., $\emptyset\not\in \cu$ and if $\cu$ contains $R\cup S$ then it contains $R$ or $S$).

Why do ultrafilters and $\{0,1\}$-ultrafilters coincide? Recall that $\eta_A(t)$ is the collection  of all $Z\subseteq A$ with $t\in Z$. And $a^{\aa}$ takes every collection of subsets to the collection of their preimages under $a$. Thus, a collection of subsets $\cu$ lies in the derived subobject $A'$ iff there exists $t\in A$ such that 
\begin{equation}\label{eq3.1}
a^{-1}(Z)\in \cu \quad \text{ iff } \quad t\in Z \qquad \; \text{ (for all $Z\subseteq A$).}
\end{equation}
We are going to prove that this holds iff $\cu$ is an ultrafilter. This can be derived from the result of Galvin and Horn \cite{GH} which states that $\cu$ is an ultrafilter iff for every finite disjoint decomposition of $X$ precisely one member lies in $\cu$. We provide a full (short) proof since we need modifications of it below.

\end{example}

\begin{lemma}\label{L3.3} Let $\ck=\Set$ with $D=\{0,1\}$. Then a $D$-ultrafilter on a set is precisely an ultrafilter on it.
\end{lemma}

\begin{proof}
To give an element  of $X^{\aa}=\cp\cp X$ means precisely to give a collection $\cu$ of subsets of $X$. It is clear that \eqref{eq3.1} holds whenever $\cu$ is an ultrafilter: in the finite decomposition $X=\cup_{t\in a[X]} a^{-1}(t)$ we have a unique $t\in A$ with $ a^{-1}(t)\in \cu$, then \eqref{eq3.1}  follows. 

Conversely, suppose $\cu$ is a $\{0,1\}$-ultrafilter. From \eqref{eq3.1} we immediately see that  $\cu\not=\emptyset$ (it contains $a^{-1}(\{t\})$).  Given subsets $R,S\subseteq X$ expressed by their characteristic functions, we put
$$A=\{0,1\}^2 \text{ and } a=\langle \chi_R,\chi_S\rangle:X\to A.$$

(i) If $R\subseteq S$ and $R\in \cu$, then $S\in \cu$. We see that $R=a^{-1}(\{(1,1)\})$, thus in \eqref{eq3.1} we have $t=(1,1)$. Consequently, $S$ lies in $\cu$, since it is $a^{-1}(Z)$ for $Z=\{(1,1),\, (0,1)\}$.

(ii) If $R,S \in \cu$, then $R\cap S\in \cu$, since this is $a^{-1}(\{(1,1)\})$.

 (iii)  If $R\cup S\in \cu$, then $R\in \cu$ or $S\in \cu$. Indeed, assuming $R=a^{-1}(\{(1,0),\, (1,1)\})$ does not lie in $\cu$, then $t$ in \eqref{eq3.1} is $(0,1)$: it cannot be $(0,0)$ since $a^{-1}(\{(0,0)\})=\emptyset$. Consequently, $S=a^{-1}(\{(0,1),\, (1,1)\})$ lies in $\cu$. And $\emptyset \not\in \cu$ since we can choose $a:X\to 1$.
\end{proof}

{
\begin{example}
For $\ck=\Par$ and $D=\{1\}$ a $D$-ultrafilter on a set is precisely an ultrafilter on it. The proof is completely analogous to Lemma \ref{L3.3}.
\end{example}
}

\begin{example}\label{E3.4} Let $\ck=\Pos$ and $D=\2$. A $D$-ultrafilter on a poset $X$ is precisely a  nonempty prime upwards closed collection $\cu$ of   $\uparrow$-sets of $X$ which is closed under  finite intersections. Here \textit{ prime} means that $\emptyset\not\in \cu$  and whenever   $R\cup S\in \cu$, then $R\in \cu$ or $S\in \cu$  (for all $\uparrow$-sets $R$, $S$).

The proof is completely analogous to that of the above lemma. To give an element of $X^{\aa}$ means, by Example \ref{E2.2}(a), to give an  upwards closed collection of $\uparrow$-sets   $\cu$. If it is nonempty, prime, and closed under finite intersections, then for every morphism $a:X\to A$ with $A$ finite, the collection  $\hat{\cu}=\{Z\in A^{\ast}\mid a^{-1}(Z) \in \cu\}$ also has those properties, thus  $\displaystyle{\cap_{Z\in \hat{\cu}}Z=\uparrow t}\; \in \hat{\cu}$ for some $t\in A$. Then $Z\in \hat{\cu}$ iff $t\in Z$, ensuring that $\cu$ is a $D$-filter. The rest is the same as in \ref{L3.3}, just the set $A=\{0,1\}^2$ is substituted by the poset $D^2$: 
$$
\begin{array}{l}\xymatrix{&(1,1)\ar@{-}[dl]\ar@{-}[dr]&\\
(1,0)\ar@{-}[dr]&&(0,1)\ar@{-}[dl]\\
&(0,0)&}\end{array}$$
\end{example}

\begin{example}\label{E3.5} Let $\ck=\JSL$ with $D=\2$. A $D$-ultrafilter on a semilattice $X$ is precisely a prime,  upwards closed collection of  prime $\uparrow$-sets of $X$ (see \ref{E2.1}\ref{jsl}). Indeed, every  element $\cu$ of $X^{\aa}$ is a $D$-ultrafilter. To see this, given a morphism   $X\xrightarrow{a}A$, put $\hat{\cu}=\{Z\in A^{\ast} \mid a^{-1}(Z)\in \cu\}$. We want to prove that there is a unique $t_0$ in $A$ such that $Z\in \hat{\cu}\; \text{ iff }\, t_0\in Z$.  If $\hat{\cu}=\emptyset$, then $t_0=0$. If $\hat{\cu}\not=\emptyset$, every $Z\in \hat{\cu}$ is of the form $Z=\uparrow u_1\cup \dots \cup \uparrow u_k$ with $u_1,\, \dots, \, u_k$ incomparable elements of $A$. Since $\cu$ is prime, so is $\hat{\cu}$, therefore some  $\uparrow u_i$ belongs to $\hat{\cu}$. Thus, there are incomparable elements of $A$, $t_1, \, \dots, \, t_n$, such that  $\hat{\cu}$ consists of all sets $\uparrow t_i$, $i=1, \dots, n$, and all sets of $A^{\ast}$ containing some of them. It is easily seen that  $t_0=t_1 \vee \dots \vee t_n$ is as desired. 

 The rest  is analogous to $\Pos$, using that the above poset $D^2$ is a semilattice.
\end{example}


\begin{example}\label{E3.6} Let $\ck=K$-$\Vec$ and $D=K$. A $D$-ultrafilter on a vector space $X$ is a vector of the double-dual space $X^{\aa}$. Indeed, for every  finite-dimensional space $A$ the unit $\eta_A:A\to A^{\aa}$ is well-known to be invertible. Thus, the derived subobject is all $X^{\aa}$.

It turns out that there is a close analogy between ultrafilters on a set and vectors of the double-dual of a space. It is based on the following observation made in \cite{ABCPR}:

(i) To give an ultrafilter on a set $X$ means precisely to give a choice, for every finite decomposition $a:X\twoheadrightarrow n$ $(n\in \mathbb{N})$ of a class $a^{-1}(i), \, i\in n$, which is \textit{compatible}. That is, if $b:X\twoheadrightarrow m$ is a coarser decomposition (one factorizing through $a$) then the chosen class for $b$ contains $a^{-1}(i)$.

(ii) To give a vector of $X^{\aa}$ for a space $X$ means precisely to give a choice, for every finite-dimensional decomposition $a:X\twoheadrightarrow K^n$ ($n\in \mathbb{N}$) of a class $a^{-1}(i)$, $i\in K^n$, which is compatible.

 A different analogy between  $X^{\aa}$ and ultrafilters was presented in 
 Devlin's thesis \cite{D}.
\end{example}

\begin{example}\label{E3.9}  Let $M$ be a finite commutative monoid and $D=\cp M$. A $D$-ultrafilter on an $M$-set $(X,\cdot)$ is precisely an ultrafilter on $X$. Indeed, we know from Example \ref{E2.1}(e) that an element of $X^{\aa}$ is a collection $\cu$ of subsets of $X$. This is a $D$-ultrafilter iff for every homomorphism $a:X\to A$ with $A$ finite (= finitely presentable) there exists $t\in A$ such that \eqref{eq3.1} holds. Therefore, every $D$-ultrafilter is an ultrafilter. The proof of the converse is analogous to Lemma \ref{L3.3}. We just use, instead of the function $\chi_R$ there, the function $g_R: X\to \cp M$ of Example \ref{E2.1}(e). Thus, we work with $a=\langle g_R,\, g_S\rangle :X\to (\cp M)^2$.
\end{example}

\begin{example}\label{E:new3.10}Let $\ck=\Gra$ and $D$ the complete graph on $\{0,1\}$. A $D$-ultrafilter on a graph is an ultrafilter on its set of vertices. This follows from the next lemma.
\end{example}

\begin{lemma}\label{L3.10}For every graph $X=(V_X, E_X)$ the intersection 
of all derived subobjects is the graph whose vertices are all 
ultrafilters on $V_X$; and ultrafilters $\mathcal{F}, \, \mathcal{G}$
form an edge iff, given $F \in \mathcal{F}$ and $G \in \mathcal{G}$, 
there exists an edge of $X$ starting 
in $F$ and ending in $G$. In symbols:
\begin{equation}\label{(11)}
E_X\cap (F\times G)\not=\emptyset\quad \text{for all $F \in \mathcal{F}$ and $G \in \mathcal{G}$.}\end{equation}
\end{lemma} 

\begin{proof}
Recall that $X^{\aa}$ is the complete graph on $\cp\cp V_X$. 
Given $a:X\to A$  with $A$ finite, a collection 
$\mathcal{F}\in \cp\cp V_X$ lies in the derived subobject $a’:A’\to X^{\aa}$ 
iff it contains the set $a^{-1}(s)$ for a unique vertex $s\in V_A$.
Thus $\mathcal{F}$ lies in all derived subobjects iff it is an 
ultrafilter – this is proved as in Lemma \ref{L3.3} using complete 
finite graphs.

The edges of $A’$ are those pairs $\mathcal{F}$, $\mathcal{G}$
  of collections of subsets of $V_X$ which $p(a):A’\to A$ maps to
 an edge of $A$. That is, for which we have
\begin{equation}\label{(22)}
(s,t) \in E_A\quad \text{where $a^{-1}(s)\in \mathcal{F}$ and $a^{-1}(t)\in \mathcal{G}$.}
\end{equation}
Thus, it is our task to prove that given ultrafilters $\mathcal{F}$ 
and $\mathcal{G}$, then \eqref{(22)} holds for all morphisms of $X/\Gra_{\sfp}$  iff \eqref{(11)} holds.

\eqref{(11)} 	$\Rightarrow$ \eqref{(22)}. Put $F=a^{-1}(s)$ and $G=a^{-1}(t)$. 
By \eqref{(11)} there is an edge $(x,y)\in E_X$ with $a(x)=s$
 and $a(y)=t$. Since $a$ is a homomorphism, this implies $(s,t)\in E_A$.

\eqref{(22)} $\Rightarrow$ \eqref{(11)}. More precisely, from the fact that \eqref{(22)} holds for all $(A,a)$ we are to derive \eqref{(11)}.
We distinguish three cases:

(i) Assuming   $F-G\in \cf$, let $\sim$ be an equivalence relation on the set $V_X$ with finitely many classes, two of which are $F-G$ and $G$.  Denote by $X/\!\sim$ the quotient graph of $X$ and by $a:X\to X/\!\sim$ the quotient map:
the vertices of $X/\!\sim$ are the equivalence classes of $\sim$, and edges are all $([x],[y])$ for edges $(x,y)\in E_X$. From $F-G\in \cf$ we get $p(a)(\cf)=F-G$; analogously,  $p(a)(\cg)=G$. Thus  we have an edge $(x,y)\in E_X$ with $x\in F-G$ and $y\in G$.

(ii) Analogously for $G-F\in \cg$.

(iii)  Suppose $\overline{F-G}$ lies in $\cf$ and $\overline{G-F}$ lies in $\cg$. Since $F\cap G=F\cap \overline{F-G}$ we conclude that $F\cap G$ lies in both $\cf$ and $\cg$. Let $\sim$ be an equivalence relation on $V_X$ with classes $F\cap G$ and $\overline{F\cap G}$ (or just one class $F\cap G=X$ in case $F=G=X$). Then for the quotient map $a:X\to X/\!\sim$  we have $p(a)(\cf)=p(a)(\cg)=F\cap G$.     Thus, $F\cap G$  is a loop of $X/\!\sim$. Consequently, in $X$ there is an edge $(x,y)\in E_X$ with $x,y\in F\cap G$.
\end{proof}

\begin{example}\label{E:new3.12}
Analogously for $\ck=\Sigma$-$\Str$ and $D$ the complete $\Sigma$-structure on $\{0,1\}$. A  $D$-ultrafilter on a structure $X$ is precisely an  ultrafilter on its underlying set. The intersection of all derived subobjects has, for every $n$-ary symbol $\sigma \in \Sigma$, the relation of those $n$-tuples of ultrafilters $\mathcal{F}_1$, $\dots$, $\mathcal{F}_n$  which fulfil 
$$\sigma_X\cap (F_1\times \dots \times F_n)\not=\emptyset\quad \text{for all $\,F_1\in \mathcal{F}_1, \, \dots, \, F_n\in \mathcal{F}_n$.}$$
\end{example}

\section{The codensity monad of the embedding $\ck_{\sfp} \hookrightarrow \ck$}\label{S4}

In this section $\ck$ is a complete, symmetric monoidal closed category with a $\ast$-cogenerator $D$.

\begin{notation}\label{N4.1}
For every object   $X$ we denote by $i_X:TX\to X^{\aa}$ the wide intersection of all derived subobjects. (Thus the underlying set of $TX$ consists of all $D$-ultrafilters on $X$.) The factorizing morphisms  are denoted by $q(a)$ for all $a:X\to A$, $A\in \ck_{\sfp}$:
\begin{equation}\label{inters}
\xymatrix{&\mbox{\color{white}\underline{\color{black}A'}}\ar@{>->}[d]^{a'}\\
TX\; \ar@{>->}[ur]^{q(a)}\ar@{>->}[r]_{i_X}&X^{\aa}}\qquad \begin{array}{l}\\ \\ \text{for $(A,a)\in X/\ck_{\sfp}$}
\end{array}\end{equation}
\end{notation}

\begin{lemma}\label{L4.2} The morphisms  $i_X:TX\to X^{\aa}$ carry a subfunctor $T$ of $(-)^{\aa}$. 
\end{lemma}

\begin{proof} The definition of $T$ on morphisms  $f:X\to Y$ follows automatically from the naturality of $i:T\to (-)^{\aa}$. Indeed, given a morphism  $f:X\to Y$, in order to verify that a (necessarily unique) morphism  $Tf$ exists making the following square
$$\xymatrix{TX\ar[r]^{i_X}\ar[d]_{Tf}&X^{\aa}\ar[d]^{f^{\aa}}\\
TY\ar[r]_{i_Y}&Y^{\aa}}$$
commutative, we just need to observe that
$f^{\aa}\ot i_X$ factorizes through all derived subojects of $Y^{\aa}$.
\noindent Indeed, for all $a:Y\to A$ with $A$ finitely presentable put 
$$\bar{a}=a\cdot f:X\to A.$$
Use the universal property of the pullback $a'$ (of $\eta_A$ along $a^{\aa}$) to define a morphism  $u$ as follows:
$$\xymatrix{TX\ar[d]_{q(\bar{a})}\ar[rd]^{i_X}&\\
\bar{A}'\ar[d]_u\ar[r]^{\bar{a}'}\ar@/_7ex/[dd]_{p(\bar{a})}&X^{\aa}\ar[d]^{f^{\aa}}\\
A'\ar[r]^{a'}\ar[d]_{p(a)}&Y^{\aa}\ar[d]^{a^{\aa}}\\
A\ar[r]_{\eta_A}&A^{\aa}}$$
Then $u\ot q(\bar{a})$ gives the desired factorization.
\end{proof}

\begin{definition}\label{R-D4.3} The functor $T$ of Lemma \ref{L4.2} carries  a monad $\mathbb{T}$ which is a submonad of $(-)^{\aa}$ via $(i_X)$. This is proved  in the next theorem. $\mathbb{T}$ is called the \textit{$D$-ultrafilter monad}. 
\end{definition}

\begin{examples}\label{E4.3}(a) For $\ck=\Set$ we see that $\mathbb{T}$ is the ultrafilter monad, for $\ck=K$-$\Vec$ it is the double-dualization monad. In both cases, $\mathbb{T}$ is well-known to be the codensity monad of $\ck_{\sfp}\hookrightarrow \ck$ (see Introduction). 

(b) In case $\ck=\Pos$ the monad $\mathbb{T}$ assigns to every poset $X$  the poset of all nonempty, prime, upwards closed  collections of  $\uparrow$-sets which are closed under  finite intersections. It is ordered by inclusion, see Examples \ref{E3.4} and \ref{E2.2}\ref{pos}.

(c) In $M$-$\Set$, $D$-ultrafilters on an object $(X,\cdot)$ are just ultrafilters on $X$. 
The monoid action assigns to every ultrafilter $\cu$ on $X$ and every element $m\in M$ the ultrafilter
$$m\cu=\{R\subseteq X;\, mR\in \cu\}$$
where
$$mR=\{x\in X;\, mx\in R\}.$$
See Examples \ref{E3.9} and \ref{E2.1}\ref{mset}.

(d) For $\ck=\Gra$, see Lemma \ref{L3.10}.  For  $\ck=\Sigma$-$\Str$, see Example \ref{E:new3.12}.

\end{examples}

\begin{theorem}\label{T4.4} Let $\ck$ be a complete, symmetric monoidal closed category with a $\ast$-cogenerator $D$. Then the $D$-ultrafilter monad is a submonad of $(-)^{\aa}$ which is the codensity monad  of the embedding $\ck_{\sfp}\hookrightarrow \ck$.
\end{theorem}

\begin{proof}
Since the natural transformation $i:T\to (-)^{\aa}$ is monic, there is at most one monad structure making $i$ a monad morphism. We are going to prove that this structure exists, and that the resulting monad fulfils, for the embedding $E:\ck_{\sfp}\to \ck$, the limit formula for codensity monads (see Introduction).

\vskip1.5mm

(i) For every object   $X$ the cone 
$$a^{\aa}: X^{\aa}\to A^{\aa} \quad (\text{for all $(A,a)\in X/\ck_{\sfp}$})$$
is collectively monic. Indeed, since $D$ is a $\ast$-object,  we have $X^{\ast}=\text{co}\hspace*{-0.4mm}\lim C^{\ast}_X$, see Notation \ref{N2.4}. Now $(-)^{\ast}:\ck^{\op}\to \ck$ is a right adjoint, thus, it takes the colimit to a  limit cone $a^{\aa}:X^{\aa}\to A^{\aa}$ in $\ck$. 

(ii) Recall the notation $p(a)$ from Definition \ref{D:derived} and $q(a)$ from Notation \ref{N4.1}. We are going to prove that  for the embedding $E:\ck_{\sfp}\to \ck$ we have  
$$TX=\lim \DX$$
with the following limit cone
\begin{equation}\label{psi}\psi_a \equiv TX\xrightarrow{q(a)}A^{\pr}\xrightarrow{p(a)}A\qquad (a\in X/\ck_{\sfp}).
\end{equation}

First, $\psi_a$ is a cone of  $\DX$, i.e., given a morphism  $h$ 
\begin{equation}\label{eq-6}\xymatrix@=1.5em{&X\ar[dl]_a\ar[rd]^b&\\
A\ar[rr]_h&&B
}\end{equation}
 of $X/\ck_{\sfp}$, then $h\ot \psi_a=\psi_b$. Indeed, the following diagram
$$\xymatrix{&&TX\ar[lld]_{q(a)}^{\hspace*{7mm}\eqref{inters}}\ar[rrd]^{q(b)}_{\eqref{inters}\hspace*{5mm}}\ar[d]^{i_X}&&\\
A^{\prime}\ar[rr]_{a^{\prime}}\ar[dd]_{p(a)}^(0.35){\hspace*{3mm}\eqref{derived}}&&X^{\aa}\ar[ld]_{a^{\aa}}\ar[rd]^{b^{\aa}}&&B^{\prime}\ar[ll]^{b^{\prime}}\ar[dd]^{p(b)}_(0.35){\eqref{derived}\hspace*{3mm}}\\
&A^{\aa}\ar[rr]^{h^{\aa}}&&B^{\aa}&\\
A\ar[ru]^{\eta_A}\ar[rrrr]_h^*++{\labelstyle (\eta \text{ natural})}&&&&B\ar[lu]_{\eta_B}}$$
commutes.

Next suppose a cone of  $\DX$ with domain $Z$ is given: 
\begin{equation*}
\begin{array}{c}
X\xrightarrow{a}A \\
 \hline \\ [-4.5mm]
Z\xrightarrow{\tilde{a}}A
\end{array} \qquad \text{for $(A,a)\in X/\ck_{\sfp}$.}
\end{equation*} 
We prove that there is a unique morphism $k$   making the following triangles 
$$\xymatrix{Z\ar[ddrr]_{\tilde{a}}\ar[rrrr]^{k}&&&&TX\ar[dl]^{q(a)}\\
&&&A'\ar[dl]^{p(a)}&\\
&&A&&}$$
commutative. The diagram $\DX^{\aa}=(-)^{\aa}\ot \DX$ has the following cone:
\begin{equation*}
\begin{array}{c}
X\xrightarrow{a}A \\
 \hline \\ [-4.5mm]
Z\xrightarrow{\tilde{a}}A\xrightarrow{\eta_A}A^{\aa}
\end{array} 
\end{equation*} 
Indeed, this is compatible with $\DX^{\aa}$, because given  a morphism  \eqref{eq-6} of $X/\ck_{\sfp}$
we have the following commutative diagram
$$\xymatrix{&&Z\ar[dl]_{\tilde{a}}\ar[dr]^{\tilde{b}}&&\\
&A\ar[dl]_{\eta_A}\ar[rr]_h&&B\ar[dr]^{\eta_B}&\\
A^{\aa}\ar[rrrr]_{h^{\aa}}&&&&B^{\aa}}$$
Since by (i) $X^{\aa}$ is the limit of $\DX^{\aa}$, we obtain a unique morphism 
$$k_0:Z\to X^{\aa}$$ 
making the following squares
$$\xymatrix{Z\ar[r]^{k_0}\ar[d]_{\tilde{a}}&X^{\aa}\ar[d]^{a^{\aa}}\\
A\ar[r]_{\eta_A}&A^{\aa}}\qquad \quad \begin{array}{l}\\ \text{for all $X\xrightarrow{a}A$ in $X/\ck_{\sfp}$}\end{array}
$$
commutative. This implies that $k_0$ factorizes through the preimage $a'$ of $\eta_A$ under $a^{\aa}$. Hence, it factorizes through $i_X=\cap a'$:
$$\xymatrix{&TX\ar[d]^{i_X}\\
Z\ar[ru]^{k}\ar[r]_{k_0}&X^{\aa}}$$
This is the desired factorization, i.e., we have
$$\psi_a\ot k=\tilde{a}\qquad \text{ for all $X\xrightarrow{a}A$ in $X/\ck_{\sfp}$}.$$
Indeed in the following diagram
$$\xymatrix{Z\ar[rrrd]^{k_0}\ar[rd]^*!/u-5pt/{\labelstyle k}\ar[rddd]_{\psi_a\ot k}&&&\\
&TX\ar[rr]^{i_X}\ar[d]_{q(a)}&&X^{\aa}\ar[dd]^{a^{\aa}}\\
&A'\ar[rru]^{a'}\ar[d]^{p(a)}&&\\
&A\ar[rr]_{\eta_A}&&A^{\aa}}$$
all inner parts commute. Thus, by using the square above we get
\begin{equation}\label{eqq}\eta_A \ot(\psi_a\ot k)=a^{\aa}\ot k_0=\eta_A \ot \tilde{a}
\end{equation}
By Proposition \ref{P2.8},  $\eta_A$ is monic, so $k$ is the desired factorization.

Given a factorization $\hat{k}$, we prove $\hat{k}=k$. Let $\hat{k}_0=i_X\ot \hat{k}$, then we get
$a^{\aa}\ot \hat{k}_0=\eta_A\ot \tilde{a}$. Comparing this with \eqref{eqq} yields $a^{\aa}\ot \hat{k}_0=a^{\aa}\ot k_0$. From (i) we conclude $\hat{k}_0=k_0$. Since $i_X$ is monic, this proves $\hat{k}=k$.

(iii)
For every morphism  $h:X\to Y$ we need to verify that the definition of $Th$ (see Lemma \ref{L4.2}) agrees with the definition in the Introduction, i.e., the triangles 
$$\xymatrix{TX\ar[dr]_{\psi_{\ba}}\ar[rr]^{Th}&&TY\ar[dl]^{\psi_a}\\
&A&}\qquad \qquad a:Y\to A \text{ in } Y/\ck_{\sfp}$$
commute for $\ba=a\ot h$. For that consider the following diagram in which we denote, for $a:Y\to A$, by $\ba^{\pr}:\bar{A}^{\pr}\to X^{\aa}$ the derived subobject of $\ba$: 
$$\xymatrix{X^{\aa}\ar@/_7ex/[dddrrr]_{\ba^{\aa}=(a\ot h)^{\aa}}\ar[rrrrrr]^{h^{\aa}}&&&&&&Y^{\aa}\ar@/^7ex/[dddlll]^{a^{\aa}}\\
&&TX\ar[llu]_{i_X}\ar[d]_*!/u5pt/{\labelstyle q(\ba)}\ar[dr]^{\psi_{\ba}}\ar[rr]^{Th}&&TY\ar[ld]_{\psi_a}\ar[d]^*!/u5pt/{\labelstyle q(a)}\ar[rru]^{i_Y}&&\\
&&\bar{A}^{\pr}\ar[lluu]^{\ba^{\pr}}\ar[r]_{p(\ba)}&A\ar[d]^{\eta_A}&A^{\pr}\ar[l]^{p(a)}\ar[rruu]_ {a^{\pr}}&&\\
&&&A^{\aa}&&&}$$
Its inner parts, except the desired triangle, commute by \eqref{derived}, \ref{N4.1}, the definition of $\psi$ and naturality of $i$. The outward triangle also commutes. Thus, the desired triangle commutes since $\eta_A$ is monic by Proposition \ref{P2.8}.

 (iv) $T$ has the structure of a monad, namely, the codensity monad of the embedding $\ck_{\sfp}\hookrightarrow \ck$. It remains to verify that it is a  submonad of $(-)^{\aa}$, more precisely, that $i:T\to (-)^{\aa}$ is a monad morphism.  We denote by $\eta^T$ and $\mu^T$ the monad structure of $T$ and by $\eta$ and $\mu$ that of $(-)^{\aa}$.
 
 To prove that $i$ preserves the unit, consider the following diagram for every object   $X$ and all $(A,a)$ in $X/\ck_{\sfp}$:
$$\xymatrix{X\ar[ddd]_a\ar[rrd]^{\eta^T_X}\ar[rrrr]^{\eta_X}&&&&X^{\aa}\ar[ddd]^{a^{\aa}}\\
&&TX\ar[lldd]_{\psi_a}\ar[d]^{q(a)}\ar[rru]^{i_X}&&&\\
&&A^{\pr}\ar[lld]^{p(a)}\ar[rruu]_{a^{\pr}}&&\\
A\ar[rrrr]_{\eta_A}&&&&A^{\aa}}$$
The left-hand triangle is the definition of $\eta^T_X$, see Introduction. All the other inner parts except the upper triangle commute by Notation \ref{N4.1},  \eqref{derived} and  \eqref{psi}. Since the outward square commutes, this proves that the upper triangle, when prolonged by $a^{\aa}$, commutes. From (i) we conclude that the triangle commutes.

To prove that $i$ preserves multiplication, recall from Introduction that $\mu_X^T$ is defined by the following commutative triangles
\begin{equation}\label{mu}\xymatrix{TTX\ar[rrrr]^{\mu^T_X}\ar[dr]_{q(\psi_a)}&&&&TX\ar[dl]^{q(a)}\\
&A^{\prime\prime}\ar[dr]_{p(\psi_a)}&&A^{\pr}\ar[dl]^{p(a)}&\\
&&A&&}\qquad \begin{array}{l}\\ \\ \\ \\ \\(A,a)\in X/\ck_{\sfp}\end{array}
\end{equation}
 Consider the desired equality
$$i_X\ot \mu_X^T=\mu_X\ot i_X^{\aa}\ot i_{TX}.$$
which in view of Notation \ref{N4.1} means
$$a^{\pr}\ot q(a)\ot \mu_X^T=\mu_X\ot (a^{\pr})^{\aa}\ot q(a)^{\aa}\ot \psi_a^{\prime}\ot q(\psi_a)$$
where the derived subobject of $\psi_a:TX\to A$ is denoted by $\psi^{\pr}_a:A^{\prime\prime} \to (TX)^{\aa}$. This follows from the commutative diagram below:
$$\xymatrix{TTX\ar@/_4ex/[dd]_{i_{TX}}\ar[d]^{q(\psi_a)}\ar[rrrr]^{\mu^T_X}&&&&TX\ar[d]_{q(a)}\ar@/^7ex/[ddddd]^{i_X}\\
A^{\prime\prime}\ar[d]^{\psi_a^{\prime}}\ar[rrr]^{p(\psi_a)}&&&A\ar[d]^{\eta_A}&
A^{\pr}\ar[l]_{p(a)}\ar[dddd]^{a^{\pr}}\\
(TX)^{\aa}\ar@/_5ex/[ddd]_{(i_X)^{\aa}}\ar[d]^{q(a)^{\aa}}\ar[rrr]^{\psi_a^{\aa}}&&&A^{\aa}\ar[lldd]^{\eta_A^{\aa}}\ar@{=}[dd]&\\
(A^{\pr})^{\aa}\ar[dd]^{(a^{\pr})^{\aa}}\ar[rrru]_{(p(a))^{\aa}}&&&&\\
&A^{\aa\aa}\ar[rr]^{\mu_A}&&A^{\aa}&\\
X^{\aa\aa}\ar[ru]_{a^{\aa\aa}}\ar[rrrr]_{\mu_X}&&&&X^{\aa}\ar[lu]_{a^{\aa}}}$$
 All inner parts commute: for the upper one see \eqref{mu},  the lowest one is the naturality of $\mu$, and the triangle above  is the monad law $\mu\ot \eta^{\aa}=\id$. All the other parts commute by definition of $a^{\pr}$ and $\psi_a^{\prime}$. Consequently, the desired outward square commutes when postcomposed by $a^{\aa}$. Once again apply (i) to see that the proof is complete.
\end{proof}

\begin{remark}\label{R:new} (1) For the limit cone $(\psi_a)$ of the above proof we have commutative squares as follows:
$$\qquad \qquad \xymatrix{TX\ar[rr]^{\psi_a}\ar[dd]_{i_X}\ar@{-->}[rd]^{q(a)}&&A\ar[dd]^{\eta_A}\\
&A^{\pr}\ar@{-->}[ld]^{a^{\pr}}\ar@{-->}[ru]_{p(a)}&\\
X^{\aa}\ar[rr]_{a^{\aa}}&&A^{\aa}}\qquad \quad \begin{array}{c}\\ \\ \\ \\ \\ \text{for }a:X\to A \text{ in }X/\ck_{\sfp}\end{array}$$
Indeed the left-hand triangle is \eqref{inters} and for the right-hand part see \eqref{derived}.

(2)  For $\ck=\Set$ the function $\psi_a:TX\to A$ takes an ultrafilter $\cf$ to the unique element $x\in A$  with $a^{-1}(x)\in \cf$.

The analogous statement is true in all our examples where $TX$ is the set of all ultrafilters on the (underlying set of ) $X$:

\centerline{$\Par$,  $M$-$\Set$, $\Gra$ and $\Sigma$-$\Str$.}

In $\Pos$ and $\JSL$ the function $\psi_a$ takes a $D$-ultrafilter $\cf$ on $X$ to the largest element $x\in A$ with $a^{-1}(\uparrow x) \in \cf$ (see Examples \ref{E3.4} and \ref{E3.5}).
\end{remark}

\begin{observation}\label{O4.5} The components $\eta^T_A:A\to TA$ of the unit of the codensity monad are invertible for all finitely presentable objects  $A$. 
Indeed, recall from  the Introduction the formula $\psi_a\ot \eta_X^T=a$. The case  $a=\id_A:A\to A$, gives
$$\psi_{\id_A}\ot \eta_A^T=\id_A.$$
On the other hand, for every $b:A\to B$ in $A/\ck_{\sfp}$, we have
$$\psi_b\ot \eta_A^T\ot \psi_{\id_A}=b\ot \psi_{\id_A}=\psi_b.$$
The morphisms $\psi_b$ are the components of a limit, therefore they are collectively monic and we get
$$\eta_A^T\ot \psi_{\id_A}=\id_A.$$
\end{observation}

\begin{corollary}\label{C4.6}The codensity monad  of the embedding $\ck_{\sfp}\hookrightarrow \ck$ is the largest submonad of $(-)^{\aa}$ whose unit has invertible components at all finitely presentable objects.
\end{corollary}

\begin{proof}We show that every submonad
$$j:(\widehat{T}, \hat{\mu},\hat{\eta})\to ((-)^{\aa}, \mu, \eta)$$
with $\hat{\eta}_A$ invertible for all $A\in \ck_{\sfp}$ factorizes through $i$. Indeed, it is sufficient to verify that for every object   $X$ and all $a:X\to A$ in $X/\ck_{\sfp}$

\centerline{$j_X$ factorizes through $a^{\pr}$.}
This implies that $j_X$ factorizes through $i_X$, i.e., we have $u_X:\widehat{T}X\to TX$ with $j_X=i_X\ot u_X$. Since $i$ and $j$ are monic monad morphisms, it follows easily that $u:\widehat{T}\to T$ is also a monad morphism.

For every $a:X\to A$ in $X/\ck_{\sfp}$ we have $\eta_A=j_A\ot \hat{\eta}_A$, thus,
$$j_A=\eta_A\ot(\hat{\eta}_A)^{-1}.$$
Since $j$ is natural, we derive from $a^{\aa}\ot j_X=j_A\ot \widehat{T}a$ that 
$$a^{\aa}\ot j_X=\eta_A\ot \hat{\eta}_A^{-1}\ot \widehat{T}a.$$
This yields the desired factorization of $j_X$ through $a^{\pr}$:

$$\xymatrix{\widehat{T}X\ar[dd]_{\widehat{T}a}\ar@{..>}[rd]\ar[rrd]^{j_X}&&\\
&A^{\pr}\ar[r]^{a^{\pr}\quad}\ar[d]^{p(a)}&X^{\aa}\ar[d]^{a^{\aa}}\\
TA\ar[r]_{\hat{\eta}_A^{-1}}&A\ar[r]_{\eta_A}&A^{\aa}}$$
\end{proof}

\begin{examples}\label{newE-JSL}
(a) The codensity monad of the embedding of finite semilattices into $\JSL$ is the (full) double-dual monad. Indeed, for every finite semilattice $A$ the dual $A^{\ast}$ is isomorphic to $A^{\op}$: to every prime $\uparrow$-set $M\subseteq A$ (see Example \ref{E2.1}\ref{jsl}) assign its meet in $A$ to get a dual isomorphism $A^{\ast}\xrightarrow{\sim}A^{\op}$. Thus $A^{\aa}$ is isomorphic to $A$, and it is easy to see that $\eta_A:A\to A^{\aa}$ is indeed an isomorphism.

(b) Analogously for $K$-$\Vec$. We thus obtain another proof of Leinster's  result that $(-)^{\aa}$ is the codensity monad.
\end{examples}

\begin{remark}\label{R4.7} Corollary \ref{C4.6} gives a characterization that does not need  the technical concept of $\ast$-cogenerator or $D$-ultrafilter.

It is an open problem whether it holds for arbitrary finitely presentable  cogenerators in arbitrary symmetric monoidal closed categories  that are locally finitely presentable.
\end{remark}

\begin{noname}Summarizing all our examples, here is a survey of  codensity monads  $\mathbb{T}$ of embeddings $\ck_{\sfp}\hookrightarrow \ck$. In each case we describe the action of $T$ on an arbitrary object   $X$; in the table we just present the  elements of the underlying set of $TX$.  The structure of $TX$ as an object   of $\ck$ follows from Examples \ref{E4.3}, Lemma \ref{L3.10} and Example \ref{E:new3.12}.  For $K$-$\Vec$, $T$ is the usual double-dual functor. In the remaining examples, for each  morphism  $f:X\to Y$ the map $Tf$ is always given by assigning to a collection $\cu\in |TX|$ of subsets of $|X|$ the collection $\{R\subseteq |Y|;\, f^{-1}(R)\in \cu\}$.

\vskip5mm

\begin{center}
\begin{tabular}{|c|c|p{8cm}|}
\hline
Category&$D$& $D$-ultrafilters on an object\\
\hline
\hline 
$\Set$&$\{0,1\}$&ultrafilters\\
\hline 
$\Par$&\begin{tabular}{p{5mm}}$\{0\}$\end{tabular}&ultrafilters  \\
\hline
$\Pos$ &\begin{tabular}{p{5mm}}\\
{\setlength{\unitlength}{0.3mm}\begin{picture}(5,10)(0,0)\put(0,0){{\small 0}$\, \bullet$}\put(0,15){{\small 1}$\, \bullet$}\put(11,3){\line(0,1){15}}\end{picture}}\\ \end{tabular}&nonempty, prime $\uparrow$-sets,  closed  upwards and under finite intersections\\
\hline
$\JSL$&\begin{tabular}{p{5mm}}\\
{\setlength{\unitlength}{0.3mm}\begin{picture}(5,10)(0,0)\put(0,0){{\small 0}$\, \bullet$}\put(0,15){{\small 1}$\, \bullet$}\put(11,3){\line(0,1){15}}\end{picture}}\\ \end{tabular}&prime collections of prime $\uparrow$-sets, closed upwards \\
\hline
$\Gra$&  \begin{tabular}{p{18mm}} \vspace{2mm} $\xymatrix{0\ar@{-}@(ul,ur)\ar@{-}[r]&1\ar@{-}@(ul,ur)}$  \\  \end{tabular}
& ultrafilters on the set of vertices \vspace{4mm} \\
\hline
$\Sigma$-$\Str$&\begin{tabular}{c}$\{0,1\}$ \\
complete\end{tabular}& ultrafilters on the underlying set\\
\hline
$M$-$\Set$&$\cp M$&ultrafilters on the underlying set\\
\hline
$K$-$\Vec$&\begin{tabular}{p{5mm}}$K$\end{tabular}&vectors of the double-dual space\\
\hline

\end{tabular}
\end{center}
\end{noname}

\vspace*{2mm}
\section{Enriched codensity monads}

 Since we work with symmetric monoidal closed categories $\ck$, it is natural to ask weather the $D$-ultrafilter monad is actually the \textit{enriched codensity monad} of the embedding $E:\ck_{\sfp}\hookrightarrow \ck$. That is: is $T$ the enriched right Kan  extension of $E$ along itself? We prove that this is indeed the case for all of our examples. For commutative varieties $\ck$ we present a general proof, the cases $\Pos$, $\Gra$ and $\Sigma$-$\Str$ are   proved individually.

\begin{remark}\label{R-enriched}(1)  Since our category $\ck$ is symmetric monoidal closed, it is enriched over itself with hom-objects given by $[A,B]$ for $A,B\in \ck$. Let us shortly recall that to say that $T$ is the ($\ck$-)enriched  right Kan extension $\mathrm{Ran}_E E$ means that for every object $X$ of $\ck$ there is an isomorphism
$$\lambda_Z: [Z,TX]\cong [\ck_{\sfp}, \ck]\big([X,E-],[Z,E-]\big)$$
 natural in $Z$ ranging over $\ck^{\op}$. The object $[\ck_{\sfp}, \ck]\big([X,E-],[Z,E-]\big)$, that we abbreviate to $N(X,Z)$ below, is the hom-object of the two enriched hom-functors in the enriched-functor category $[\ck_{\sfp},\ck]$.

(2)  Since $\ck$ is complete, the object $N(X,Z)$ can be described by an equalizer in $\ck$, see \cite{Kelly} or  \cite[Proposition 6.3.1]{Borceux}.  We define below, for every pair $B$, $C$ in  $\ck_{\sfp}$, morphisms
\begin{equation*}\label{pair}\xymatrix{\displaystyle{\prod_{A\in \ck_{\sfp}}}[[X,A],[Z,A]]\ar@<1ex>[rr]^{u_{B,C} }\ar@<-1ex>[rr]_{v_{B,C} }&&\big[[B,C],\, [[X,B],[Z,C]]\big]}
\end{equation*}
such that $N(X,Z)$ is given by the following equalizer:
\begin{equation}\label{equal}\xymatrix{N(X,Z)\ar[r]^{e\qquad}&\displaystyle{\prod_{A\in \ck_{\sfp}}}[[X,A],[Z,A]]\ar@<1ex>[rr]^{\langle u_{B,C}\rangle \qquad  }\ar@<-1ex>[rr]_{\langle v_{B,C}\rangle \qquad }&&\displaystyle{\prod_{B,\, C\in \ck_{\sfp}}}\big[[B,C],\, [[X,B],[Z,C]]\big]}
\end{equation}
The above morphism $u_{B,C}$ is the transpose of the following morphism $\hat{u}_{B,C}$
\begin{equation}\label{def-u}
\xymatrix{
\big(\prod_{A\in \ck_{\sfp}} \big[[X,A],[Z,A]\big]\big)\otimes [B,C]      \ar[rrd]_{\hat{u}_{B,C}}    \ar[rr]^{ p_B\otimes [Z,E-]_{B,C}  }   
                                                                                                               & &  \big[[X,B],[Z,B]\big]\otimes \big[[Z,B],[Z,C]\big]\ar[d]^{c}\\                                                                                                                &&                            
 \big[[X,B],[Z,C]\big]\big]
}%
\end{equation}
where $p_B$ is the projection and $c$ is the composition map.
And the morphism $v_{B,C}$ is the transpose of  $\hat{v}_{B,C}$ below starting with the symmetry isomorphism $s$:
\begin{equation}\label{def-v}
\xymatrix{
\big(\prod_{A\in \ck_{\sfp}} \big[[X,A],[Z,A]\big]\big)\otimes [B,C]      \ar[rrdd]_{\hat{v}_{B,C}}       \ar[rr]^s     &&    [B,C]\otimes \big(\prod_{A\in \ck_{\sfp}} \big[[X,A],[Z,A]\big]\big)\ar[d]^{[X,E-]_{B,C}\otimes p_{C}}\\
                                                                                                               & &                          \big[[X,B],[X,C]\big]\otimes \big[[X,C],[Z,C]\big]\ar[d]^{c}\\
                                                                                                                 &&                             \big[[X,B],[Z,C]\big]
}%
\end{equation}

(3) We thus obtain a functor 
$$N(X,-):\ck^{\op}
\to \ck$$
where the action on morphisms $h:Z\to \overline{Z}$ of $\ck^{\op}$ is specified by  the equalizer \eqref{equal}: 
 let $\overline{e}$, $\overline{u}_{A,B}$ and $\overline{v}_{A,B}$ denote the above morphisms related to $\overline{Z}$ in place of $Z$. Then $N(X,h)$ is the unique morphism making the diagrams below commutative:
\begin{equation}\label{equal2}\xymatrix{N(X,Z)\ar[r]^{e\qquad}\ar[d]_{N(X,h)}&\prod_A[[X,A],[Z,A]]\ar@<1ex>[rr]^{\langle u_{B,C}\rangle \quad}\ar@<-1ex>[rr]_{\langle v_{B,C}\rangle \quad }\ar[d]_{\prod_A[[X,A],[h,A]]}&&\prod_{B,C}\big[[B,C],[[X,B],[Z,C]]\big]\ar[d]^{\prod_{B,C}[[B,C],[[X,B],[h,C]]}\\
N(X,\overline{Z})\ar[r]^{\overline{e}\qquad}&\prod_A[[X,A],[\overline{Z},A]]\ar@<1ex>[rr]^{\langle \overline{u}_{B,C}\rangle \quad}\ar@<-1ex>[rr]_{\langle \overline{v}_{B,C}\rangle \quad }&&\prod_{B,C}\big[[B,C],[[X,B],[\overline{Z},C]]\big]}
\end{equation}

(4) In each of our examples we are going to define an isomorphism $\lambda_Z: [Z,TX]\to N(X,Z)$ and prove that it is natural, in the ordinary sense, in $Z\in \ck^{\op}$. This natural  transformation is automatically enriched (see \cite[Proposition 6.2.8]{Borceux}).
\end{remark}


\begin{theorem}\label{T:var} The $D$-ultrafilter monad on a commutative variety is the enriched codensity monad of the embedding of finitely presentable algebras.
\end{theorem}

\begin{proof}(1) A description of $N(X,Z)$. Recall that, for two algebras $B$ and $C$ in a commutative variety $\ck$, the hom-object $[B,C]$ is the set of all homomorphisms from $B$ to $C$  with operations defined pointwise.

The codomain of $e$ in \eqref{equal} is thus the algebra of all families of homomorphisms 
$$\tau=(\tau_A)_{A\in \ck_{\sfp}}, \quad \tau_A:[X,A]\to [Z,A]$$
with pointwise operations: given an $n$-ary operation $\sigma$ in the given signature and $n$ families $\tau^0, \dots, \tau^{n-1}$, then the family $\tau=\sigma(\tau^0, \dots, \tau^{n-1})$ is given by
\begin{equation}\label{ope}
\tau_A(a)=\sigma_{[Z,A]}\big(\tau_A^0(a), \dots, \tau_A^{n-1}(a)\big)\quad \text{for all $a:X\to A$,  $A\in\ck_{\sfp}$.}
\end{equation}
Given algebras $B,\, C\in \ck_{\sfp}$, then the morphism $\hat{u}_{B,C}$ of \eqref{def-u} takes a  collection $\tau=(\tau_A)$ and a homomorphism $f:B\to C$ to the homomorphism 
$$\hat{u}_{B,C}(\tau,f):[X,B]\to [Z,C]$$
given by $(X\xrightarrow{b}B) \mapsto (Z\xrightarrow{\tau_B(b)}B\xrightarrow{f}C)$. Consequently, the transpose $u_{B,C}$ takes a family $\tau$  to the  homomorphism
$u_{B,C}(\tau): [B,C]\to [[X,B],[Z,C]]$ given by 
$${u_{B,C}}(\tau)\, :\, (B\xrightarrow{f}C)\xymatrix{\ar@{|->}[r]&} \big([X,B]\xrightarrow{\tau_B}[Z,B]\xrightarrow{[Z,f]}[Z,C]\big).$$
Whereas $v_{B,C}$ takes $\tau$ to the homomorphism $v_{B,C}(\tau)$ given by 
$${v_{B,C}}(\tau)\, :\, (B\xrightarrow{f}C)\xymatrix{\ar@{|->}[r]&} \big([X,B]\xrightarrow{[X,f]}[X,C]\xrightarrow{\tau_C}[Z,C]\big).$$
The equalizer $e:N(X,Z)\hookrightarrow \prod_{A\in \ck_{\sfp}} [[X,A], [Z,A]]$ is  thus  the subalgebra of the product given by all families $\tau$ for which the following squares
$$\qquad\qquad\qquad \qquad \quad \begin{array}{c}\xymatrix{[X,B]\ar[d]_{[X,f]}\ar[rr]^{\tau_B}&&[Z,B]\ar[d]^{[Z,f]}\\
[X,C]\ar[rr]_{\tau_C}&&[Z,C]}\end{array}\qquad ( f:B\to C \text{ in } \ck_{\sfp})$$
commute. In other words: 
$N(X,Z)$ consists of all ordinary natural transformations
$\tau:[X,-] \to [Z, -]$.
The operations of the algebra $N(X,Z)$ are given by \eqref{ope}.

(2) The definition of the morphism $\lambda_Z: [Z,TX]\to N(X,Z)$. Recall from the proof of Theorem \ref{T4.4} that $TX$ is the limit of the diagram
$$C_X:X/\ck_{\sfp}\to \ck, \qquad (X\xrightarrow{a}A)\mapsto A$$
with the limit cone $\psi_a:TX\to A$. We first verify that  for every finitely presentable algebra $A$  the map $a\mapsto \psi_a$ is a homomorhism from $[X,A]$ to $[TX,A]$. Thus suppose that $\sigma$ is an $n$-ary operation in our signature and we are given
$$a=\sigma_{[X,A]}(a_0,\dots, a_{n-1}).$$
Then we prove that  $$\psi_a=\sigma_{[TX,A]}(\psi_{a_0},\dots, \psi_{a_{n-1}}).$$
We use the fact that $\mathbb{T}$ is a submonad of the double-dualization monad via $i_X:TX\to X^{\aa}$, see Theorem \ref{T4.4}. Since $(-)^{\aa}$ is an enriched functor, we have 
$$a^{\aa}=\sigma_{[X,A]^{\aa}}(a^{\aa}_0,\dots,a^{\aa}_{n-1}).$$ 
From the fact that operations are defined pointwise we derive
$$a^{\aa}\ot i_X=\sigma_{[TX, A^{\aa}]}(a_0^{\aa} \cdot i_X, \dots, a_{n-1}^{\aa} \cdot i_X).$$
Consequently, for the unit $\eta_A:A\to TA$ we get
$$\begin{array}{rll}
\eta_A\ot \psi_a         &            =a^{\aa}\ot i_X                                                                   &\text{by Remark \ref{R:new}}    \\
                             &= \sigma_{[TX,A^{\aa}]}(a^{\aa}_0\ot i_X, \dots, a^{\aa}_{n-1}\ot i_X) &                                              \\
                             &=\sigma_{[TX, A^{\aa}]}(\eta_A\ot\psi_{a_0},\dots, \eta_A\ot \psi_{a_{n-1}})&    \text{by Remark \ref{R:new}}                                            \\
                                 &=  \eta_A\ot  \sigma_{[X,A]}(\psi_{a_0},\dots, \psi_{a_{n-1}})            & \text{since $\eta_A$ is a homomorphism.}\\
\end{array}$$
We know from Proposition \ref{P2.8} that $\eta_A$ is monic, which establishes the desired equality.

Consequently, every homomorphism $z:Z\to TX$ yields a natural transformation
$$\tau^z:[X,E-]\to [Z,E-]$$
whose component $\tau^z_A$ for $A\in \ck_{\sfp}$ is defined by
$$\tau^z_A: (X\xrightarrow{a}A) \mapsto (Z\xrightarrow{z}TX\xrightarrow{\psi_a}A).$$
Indeed, each $\tau^z_A$ is a homomorphism, since given $a=\sigma_{[X,A]}(a_0,\dots,a_{n-1})$ as above, we have
$$
\psi_a\cdot z=\sigma_{[TX,A]}(\psi_{a_0},\dots, \psi_{a_{n-1}})\cdot z=\sigma_{[Z,A]}(\psi_{a_0}\cdot z,\dots, \psi_{a_{n-1}}\cdot z)$$
(again using the pointwise definition of $\sigma$). And the naturality of $\tau^z$ follows from the fact that $(\psi_a)$ is a cone of the diagram $C_X:X/\ck_{\sfp}\to \ck$:
for every morphism of $X/\ck_{\sfp}$
%
$$\xymatrix{&X\ar[ld]_a\ar[rd]^{a'}&\\A\ar[rr]_u&&A'}$$
we have $\psi_{a'}=u\ot \psi_a$, hence, $\psi_{a'}\ot z=u\ot(\psi_a\ot z)$. 

We define $\lambda_Z:[Z,TX]\to N(X,Z)$ by 
\vskip1mm
\centerline{$\lambda_Z(z)=\tau^z$ for all $z:Z\to TX$.}
\vskip1mm
(3)  $\lambda_Z$ is an isomorphism. First, the underlying function $z\mapsto \tau^z$ is bijective. This follows from the universal property of the limit cone  $(\psi_a)$, since  to give a cone of $C_X$ with domain $Z$ is nothing else then to give a natural transformation $\tau$  from $\ck(X,E-)$ to $\ck(Z, E-)$.

It remains to verify that $\lambda_Z$ is a homomorphism. Given elements
 $z_0$, ..., $z_{n-1}$ of $[Z,TX]$, let $z=\sigma_{[Z,TX]}(z_0,\dots,z_{n-1})$. We verify that 
$$\tau^z=\sigma_{N(X,Z)}\big(\tau^{z_0},\dots, \tau^{z_{n-1}}\big)$$
where the right-hand side is the above operation  \eqref{ope} on natural transformations. That is, the $A$-component of the right-hand side is given, in each $a$ of $[X,A]$, by
$$\sigma_{[Z,A]}\big(\psi_a\ot z_0,\dots,\psi_a\ot z_{n-1})=\psi_a\ot \sigma_{[Z,TX]}(z_0,\dots,z_{n-1}).$$
And this is precisely the $a$-component of $\tau^z_A$.

(4) The naturality of $\lambda$. Given a homomorphism $h:\overline{Z}\to Z$, the homomorphism $N(X,h)$ of \eqref{equal2}   takes a natural transformation $\tau:[X,E-]\to [Z,E-]$ to the natural transformation $\widehat{\tau}$ with components
$$\xymatrix{\widehat{\tau}_A: [X,A]\xrightarrow{\tau_A}[Z,A]\xrightarrow{[h,A]}[\overline{Z},A]}$$
Thus, $N(X,h)\ot \lambda_Z$ assigns to each $z:Z\to TX$ the natural transformation $\widehat{\tau^z}$ with components
$$\widehat{\tau^z}_A(a)=[h,A](\psi_a\ot z)=\psi_a\ot z\ot h.$$
And $\lambda_{\ov{Z}}\ot[h,TX]$ assigns to it   $\tau^{z\ot h}$ with the same components.
\end{proof}


\begin{example}\label{E5:Pos} The $D$-ultrafilter monad on $\ck=\Pos$ is the enriched codensity monad of the embedding of finite posets. The proof is analogous to the preceding one:

(1) A description of the poset $N(X,Z)$. Recall that for posets $B$, $C$ we have the poset $[B,C]$ of all monotone functions from $B$ to $C$ ordered componentwise. Arguing precisely as in the above proof we conclude that  
$N(X,Z)$ consists of all ordinary  natural transformations $\tau:[X,E-]\to [Z,E-]$ ordered componentwise:
$$\tau\leq \tau'\text{ iff } \tau_A(a)\leq \tau^{\pr}_A(a) \text{ for all } A\in \Pos_{\sfp} \text{ and } (A,a) \in X/\Pos_{\sfp}.$$

(2) The definition of the morphism $\lambda_Z:[Z,TX]\to N(X,Z)$. We first verify that for every finite poset $A$ we have a monotone map
$$[X,A]\to[TX,A], \quad a\mapsto \psi_a.$$
For that recall the limit cone $\psi_a:TX\to A$: it takes a $D$-ultrafilter $\cf$ on $X$ to the largest element $t\in A$ with $a^{-1}(\uparrow t)\in \cf$ (see Remark \ref{R:new}). Given $a\leq b$ in $[X,A]$, then for the above $t$ we have
$$a^{-1}(\uparrow t)\subseteq b^{-1}(\uparrow t), \text{ thus, } b^{-1}(\uparrow t)\in \cf .$$
This implies for $s=\psi_b(\cf)$ that $t\leq s$. Shortly, $\psi_a(\cf)\leq \psi_b(\cf)$.

The rest is analogous to the proof above, part (2): we define $\lambda_Z(z)=\tau^z$ with components $\tau^z_A(a)=\psi_a\ot z$, and we have that each $\tau^z$ is a natural transformation. 

(3) $\gl_Z$ is an  isomorphism. Since the underlying function is bijective, we only need to prove for all $A$ finite and all monotone $z,\, \ov{z}:Z\to TX$ that 
$$z\leq \ov{z} \quad \text{iff}\quad \tau^z\leq \tau^{\ov{z}}.$$
Indeed, for $z\leq \ov{z}$ we derive $\tau^{z}(a)=\psi_a\ot z\leq \psi_a\ot \ov{z}= \tau^{\ov{z}}(a)$, i.e., $ \tau^z\leq \tau^{\ov{z}}$. Conversely, if $ \tau^z\leq \tau^{\ov{z}}$, then for all $a$ we have $\psi_a \ot z\leq \psi_a \ot \ov{z}$. Since limits in $\Pos$ are conical this implies $z\leq \ov{z}$.

(4) The proof of naturality of $\gl$ is completely analogous to the above proof. 
\end{example}

\begin{example}\label{exa4.14} The $D$-ultrafilter monad on $\Gra$ is the enriched codensity monad of the embedding of finite graphs.

(1) The description of the graph $N(X,Z)$. 
Recall that the hom-object $[B,C]$ for graphs $B$ and $C$ has as vertices all functions  $r:VB\to VC$, where $V:\Gra\to \Set$ is the usual forgetful functor. And edges are pairs $(r,r')$ of functions such that
$$(x,x')\in E_B\quad \text{implies} \quad (r(x),r'(x'))\in E_C.$$
The codomain of the equalizer \eqref{equal} defining $N(X,Z)$ is the graph $\prod_{A\in \Gra_{\sfp}}[[X,A],[Z,A]]$ of all collections $\tau=(\tau_A)$ of functions
$$\tau_A:\Set(VX,VA)\to \Set(VZ,VA).$$
Two such collections $(\tau,\tau')$ form an edge of the product iff every projection to $[[X,A], [Z,A]]$ yields an edge $(\tau_A, \tau^{\pr}_A)$, i.e., iff the following implication holds:
\begin{equation}\label{edge}
(r,r')\in E_{[X,A]} \quad \text{implies}\quad (\tau_A(r), \tau^{\pr}_A(r'))\in E_{[Z,A]} 
\end{equation}
for every finite graph $A$.

The morphism  $u_{B,C}$  takes a family $\tau$ to the function
$$u_{B,C}(\tau):[B,C]\to [[X,B],[Z,C]]$$
defined by
$$(VB\xrightarrow{r}VC)\mapsto [Z,r]\ot \tau_B=\Set(VZ,r)\ot \tau_B$$
whereas  $v_{B,C}$ takes $\tau$ to the function $v_{B,C}(\tau)$ defined by
$$(VB\xrightarrow{r}VC)\mapsto \tau_C\ot [X,r]=\tau_C\ot \Set(VX,r).$$
The equalizer $N(X,Z)$ is then given by all families $\tau$  such that   for every function $r:VB\to VC$ with $B$, $C$ finite the square below
$$\xymatrix{\Set(VX,VB)\ar[d]_{\Set(VX,r)}\ar[r]^{\tau_B}&\Set(VZ,VB)\ar[d]^{\Set(VZ,r)}\\
\Set(VX,VC)\ar[r]_{\tau_C}&\Set(VZ,VC)}$$
commutes. In other words,
$N(X,Z)$ consists of all  natural transformations 
$$\tau: \Set(VX,VE-)\to \Set(VZ,VE-).$$
The edges of $N(X,Z)$ are pairs $\tau$, $\tau'$ of  natural transformations  satisfying \eqref{edge}.

(2) The definition of the morphism  $\gl_Z:[Z,TX]\to N(X,Z)$.  
Let $C_X:X/\Gra_{\sfp}\to \Gra$ and $C_{VX}: VX/\Set_{\sfp}\to \Set$ be the canonical diagrams. Thus $C_{VX}$ takes a map $r:VX\to M$, $M$ a finite set, to $M$. We denote its  limit cone by
\begin{equation*}\widetilde{T}VX\xrightarrow{\widetilde{\psi}_r} M\, , \; \text{ for all }r:VX\to M, \, M \text{ finite. }\end{equation*}
 The following triangle
$$\xymatrix{X/\Gra_{\sfp}
\ar[r]^{C_X}\ar[d]_H&
\Gra\ar[r]^V&\Set\\
VX/\Set_{\sfp}\ar[rru]_{C_{VX}}}$$
where $H$ takes every $X\xrightarrow{a}A$ to $VX\xrightarrow{Va}VA$, is commutative. Since $H$ is a final functor and $V$ preserves limits, we obtain
$$\wt{T}X=\lim C_{VX}=\lim VC_X =V\lim C_X=VTX$$
with the limit cone $(\wt{\psi}_r)_{r\in VX/\Set_{\sfp}}$ satisfying $V\psi_a=\wt{\psi}_{Va}$ for every $a\in X/\Gra_{\sfp}$.

Cones $\gamma$ from $VZ$ to $C_{VX}$ are in bijective correspondence with cones from $VZ$ to $VC_X$ via the assignment 
$(\gamma_r)_{r\in VX/\Set_{\sfp}} \, \mapsto \, (\gamma_{Va})_{a\in X/\Gra_{\sfp}}.$
And the latter  are in bijective correspondence  with the natural transformations from $\Set(VX,VE-)$ to $\Set(VZ,VE-)$. For every function $z\in [Z,TX]$ we have 
the  natural transformation  $\tau^z:\Set(VX,VE-)\to \Set(VZ,VE-)$ with components
$$\tau^z:(VX\xrightarrow{r}VA)\mapsto (VZ\xrightarrow{r}VTX\xrightarrow{\wt{\psi}_{r}}VA).$$
It follows that we can define the function $\lambda_Z$ as before by $\lambda_Z(z)=\tau^z$, 
and, moreover, that it is a bijection.

(3) We next verify that the map $r\mapsto \wt{\psi}_{r}$ is a graph morphism  from $[X,A]$ to $[TX,A]$ for every finite graph $A$. That is, we verify that
$$(r,r')\in E_{[X,A]} \quad \text{implies} \quad (\wt{\psi}_r, \wt{\psi}_{r'})\in E_{[TX,A]}.$$
Thus we must prove that
$$(r,r')\in E_{[X,A]} \; \text{and}\; (\cf,\cf')\in E_{TX}\quad \text{imply} \quad (\wt{\psi}_r(\cf), \wt{\psi}_{r'}(\cf'))\in E_{A}.$$
Put $u=\wt{\psi}_r(\cf)$, this is the unique vertex of $A$ with $r^{-1}(u)\in \cf$, see Remark \ref{R:new}. Analogously $u'=\wt{\psi}_{r'}(\cf')$ yields $(r')^{-1}(u')\in \cf'$. Recall the description of edges of $E_{TX}$ in Lemma \ref{L3.10}.
(4) $\gl_Z$ is an isomorphism. We already saw in (2) that it is a bijection. 
It remains to prove that $\gl_Z$ and its inverse preserve edges.

(4a) $\gl_Z$ preserves edges. Let $(z,z')\in E_{[Z,TX]}$, we are to prove that the pair $(\tau^z,\tau^{z'})$ satisfies \eqref{edge} above:
$$\text{if}\;\; (r,r')\in E_{[X,A]} \quad \text{then} \quad  (\wt{\psi}_r \ot z, \wt{\psi}_{r'}\ot z')\in E_{[Z,A]}.$$
That is, for every edge $(u,u')\in E_Z$ we are to verify 
$$(\wt{\psi}_r \ot z(u), \wt{\psi}_{r'}\ot z'(u'))\in E_{A}.$$
The ultrafilters $\cf=z(u)$ and  $\cf'=z'(u')$  form an edge of $TX$ due to $(z,z')\in E_{[Z,TX]}$. From (3) above we get that $(\wt{\psi}_r(\cf), \wt{\psi}_{r'}(\cf'))$ is an edge, as desired.

(4b) $\gl_Z^{-1}$ preserves edges. In other words, for arbitrary $z,z':VZ\to VTX$ we are to prove that 
\vskip1mm
\centerline{$(\tau^z,\tau^{z'})\in E_{TX}$ implies $(z,z')\in E_{[Z,TX]}$.}

\vskip1mm

\noindent  That is, we should  prove
$$(z(u),z'(u'))\in E_{TX}\quad \text{for all $(u,u')\in E_Z$.}$$
By \eqref{edge}, $$\text{$(\tau_A^z(r),\tau_A^{z'}(r'))\in E_{[Z,A]}$ for all $(r,r')\in E_{[X,A]}$.}$$
 For every $(X\xrightarrow{a}A)$ in $X/\Gra_{\sfp}$, we know from Example \ref{E2.2}(b) that $Va$ is a loop of $[X,A]$, hence 
the pair $\big(\tau_A^z(Va),\tau_A^{z'}(Va)\big)$ is an edge of  ${[Z,A]}$. Therefore, if $(u,u')\in E_Z$, we have that $(\wt{\psi}_{Va}\ot z(u),\wt{\psi}_{Va}\ot{z'}(u'))=(V\psi_a(z(u)),V\psi_a({z'}(u')))$ belongs to  $E_A$ for all $a\in X/\Gra_{\sfp}$; consequently, $(z(u),z'(u'))\in E_{TX}$, see Lemma \ref{L3.10}.


(4)  The naturality of $\gl$ is shown analogously to the  proof of \ref{E5:Pos}.
\end{example}

\begin{example} For the category  $\Sigma$-$\Str$ the $D$-ultrafilter monad  is also the enriched codensity monad of the embedding of finite structures. The details are   completely analogous to the case of graphs above.
\end{example}

\vspace*{10mm}

\section{Further Examples}\label{S5}

In this section we consider a more general setting: a complete category  $\ck$ and a small, full subcategory $\ca$. We discuss the codensity monad of the embedding $\ca\hookrightarrow  \ck$.

Given a set $\{D_i\}_{i\in I}$ of cogenerators of $\ck$ lying in $\ca$ we obtain a monad $\mathbb{S}$ on $\ck$ from the well-known adjunction $L\dashv R: \big(\Set^{I}\big)^{\op} \to \ck$ where
$$LX=\big(\ck(X, D_i)\big)_{i\in I} \; \text{ and } \; R(M_i)_{i\in I}=\prod_{i\in I}  D_i^{M_i}.$$
We are going to characterize the codensity monad of $\ca \hookrightarrow \ck$ as the smallest submonad of $\mathbb{S}$ with a property called the \textit{limit property} below.
 We continue using the notation of Introduction:
$$C_X:X/\ca\to \ck, \; (X\xrightarrow{a}A) \mapsto A.$$
 
 \begin{remark}\label{R5.1} The above monad $\mathbb{S}$ is given on objects $X$ by  $SX=\prod_{i\in I}D_i ^{\ck(X,D_i)}$ with the unit $\eta^S:\Id\to S$ defined by the projections $\pi_f$ for $ f:X\to D_i$, as follows
 $$\pi_f\ot \eta_X^S=f.$$
 Thus $\eta_X^S$ is monic, since $(D_i)$ is a cogenerating set.
 
 The multiplication $\mu^S$ is determined by the commutativity of the following  triangles
 $$\xymatrix{SSX\ar[rr]^{\mu^S_X}\ar[dr]_{\pi_{\pi_a}}&&SX\ar[dl]^{\pi_a}\\
 &D_i&}$$
 for all $a:X\to D_i$ and $i\in I$.
 \end{remark}
 
 \begin{definition}\label{Dlimit}
 A monad $\mathbb{T}$ on $\ck$ has the \textit{ limit property} (with respect to the embedding $\ca\hookrightarrow \ck$) if for every object $X$ we have $TX=\lim TC_X$ with the canonical limit cone of all $Ta$ for $a\in X/\ca$.
 \end{definition}
 
 \begin{example} (1) The codensity monad of $\ca \hookrightarrow \ck$ has the limit property: use the limit formula. 
 
 (2) In a symmetric monoidal closed complete category $\ck$, for every $\ast$-object $D$, the double-dualization monad $(-)^{\aa}=[[-,D],D]$ has the limit property, since $[-,D]:\ck^{\op}\to \ck$ is a right adjoint.
 \end{example}
 
 \begin{lemma}\label{Llim} The monad $\mathbb{S}$ has the limit property.
\end{lemma}

\begin{proof}Since $S=R\ot L$ and $R$ preserves limits, it is sufficient to prove that the diagram $L\ot C_X$  in $\big(\Set^I\big)^{\op}$ has the limit $\big(\ck(X,D_i)\big)_{i\in I}$ with the limit cone of all maps $(-)\ot a:\big(\ck(A,D_i)\big)_{i\in I}\to \big(\ck(X,D_i)\big)_{i\in I}$ for   $a:X\to A$. We can work with the components individually, thus, let $i\in I$ be fixed. Hence in $\Set$, rather than $\Set^{\op}$, we are to prove that the cocone
 $$\ck(A,D_i)\xrightarrow{(-)\ot a}\ck(X, D_i) \quad \quad \text{( $(a,A)\in X/\ca$)}$$
 is a colimit   of $\ck(E-,D_i):(X/\ca)^{\op}\to \Set$. Indeed, let 
 $$z_a:\ck(A,D_i)\to Z\quad \quad \text{( $(a,A)\in X/\ca$)}$$
 be  another cocone of $\ck(E-,D_i)$. Compatibility means that given a morphism in $X/\ca$
 $$\xymatrix{&X\ar[dl]_a\ar[dr]^b&\\
 A\ar[rr]_u&&B}$$
then $z_b(t)=z_a(tu)$ for all $t:B\to D_i$.
 The function 
 $$z:\ck(X,D_i)\to Z, \; \; \; z(t)=z_t(\id_{D_i}),$$
 for all $t:X\to D_i$, is the desired factorization.
 
 Indeed, the equality $z_a=z\ot \big((-)\ot a\big)$ means that 
 $$z_a(r)=z(r\ot a)=z_{ra}(\id_{D_i})\qquad \text{ for all $r:A\to D_i$}$$
 by choosing $t=\id_{D_i}$ and $u=r$ (thus $b=ra$).
 
 The uniqueness of $z$ is clear.
\end{proof}

 
 \begin{theorem}\label{T5.4} The codensity monad of the embedding $\ca\hookrightarrow \ck$ is the smallest submonad of $\mathbb{S}$ with the limit property.
 \end{theorem}
 
 \begin{proof} (1) Let $\mathbb{M}$ be a monad on $\ck$ with the limit property and with a monic unit $\eta:\Id\to \mathbb{M}$. Looking at the proof of Theorem \ref{T4.4}, we see that it works for  $\ca\hookrightarrow \ck$  if, instead of the double-dualization monad $(-)^{\aa}$, we take the monad $\mathbb{M}$. Thus, the codensity monad can be obtained from $\mathbb{M}$ by using the intersection of derived subobjects analogous to that described in Definition \ref{D:derived} and  Lemma \ref{L4.2}. In particular, the codensity monad is a submonad of $\mathbb{M}$.
For $\mathbb{M}=\mathbb{S}$, we deduce that the codensity monad $\mathbb{T}$ is a submonad of $\mathbb{S}$.

 (2) Let $\mathbb{T}$ be the monad defined analogously to Theorem \ref{T4.4} with $\mathbb{S}$ replacing $(-)^{\aa}$ everywhere. Thus, 
 for every object   $X$,  $TX$ is the intersection of the preimages of $\eta_A^S$ (see Remark \ref{R5.1}) under $Sa$ for all $a:X\to A$ in $X/\ck$:
\begin{equation}\label{eqA0}\xymatrix{TX\ar[d]_{q(\bar{a})}\ar[rd]^{i_X}&\\
A_0\ar[r]^{a_0}\ar[d]_{p(a)}&SX\ar[d]^{Sa}\\
A\ar[r]_{\eta^S_A}&SA}
\end{equation}
 This defines a functor $T$, its action on morphisms  is defined precisely as in Lemma \ref{L4.2}.
 
 Then $\mathbb{T}$ is a submonad of $\mathbb{S}$ via the monad morphism $i: \mathbb{T}\to \mathbb{S}$ 
 with the above components $i_X$. 
 
 (3)  Moreover, this works in a entirely similar way for every submonad $\overline{\mathbb{S}}$ of $\mathbb{S}$ with the limit property, showing that the codensity monad is a submonad of any such $\overline{\mathbb{S}}$.
 
 Since the codensity monad   has the limit property,  the proof is concluded.
 \end{proof}

 \begin{example}\label{E5.5} Let $\ck$ be a locally finitely presentable  category  with a cogenerating set $(D_i)_{i\in I}$ in  $\ck_{\sfp}$.  Then the codensity monad  of the embedding of $\ck_{\sfp}$ into $\ck$ is the smallest submonad of the monad $SX=\prod_{i\in I}D_i ^{\ck(X,D_i)}$ with the limit property. This is actually quite analogous to the description of Section \ref{S4}, just the desired subobjects are now related to $\mathbb{S}$ rather than  $(-)^{\aa}$ (see the proof above). However, in the concrete situations of Section \ref{S4} the description using $\ast$-cogenerators is more illustrative. 
 \end{example}
 
 Given a $\ast$-cogenerator $D$, how is the present description related to that of Theorem \ref{T4.4}?  We would like to see the codensity monad  of Section \ref{S4} as a submonad of $\mathbb{S}$ with the limit property. For that we need $(-)^{\aa}$ to be a submonad of $\mathbb{S}$. This holds for the examples of Section \ref{S4}, as we are going to show.
 
 \begin{remark}\label{D-e} Let $\ck$ be a complete, symmetric monoidal closed category with a $\ast$-cogenerator $D$. 

(1) For every object $X$ the morphism $(\eta_X)^{\ast}: X^{\ast\ast\ast} \to X^{\ast}$ yields an algebra for the monad $(-)^{\aa}$. In particular, since $D$ is isomorphic to $I^{\ast}$, we obtain such an algebra $D$ that we denote by
 		$$e: D^{\aa} \to D.$$
That is, if $i: D \to I^{\ast}=[I,D]$ denotes the canonical isomorphism, then $e=i^{-1}\cdot  \eta_I^{\ast}\cdot i^{\aa}$.

 	(2) In the next result we assume  the morphisms $e\ot a^{\aa}$, $a\in X/\ck_{\sfp}$, to be  jointly monic. This holds in all our examples of Section 4. Indeed there we have
 	$$e=\eta_{D^{\ast}}(\id_D):D^{\aa}\to D.$$

 	(3) Denote by $\pi_a:SX \to D$ the projection of $SX = D^{\ck(X,D)}$  corresponding to $a:X\to D$.
We can define a unique morphism $m_X$ by the following commutative squares:
 $$\xymatrix{X^{\aa}\ar[d]_{a^{\aa}}\ar[r]^{m_X}&SX\ar[d]^{\pi_a}\\
D^{\aa}\ar[r]_{e}&D}$$
\end{remark}

\begin{lemma}
 	Let $\ck$ be a complete, symmetric monoidal closed category with a $\ast$-cogenerator $D$, and let $\A = \ck_{\sfp}$. Assuming that the morphisms $e\ot a^{\aa}$, $a\in X/\ck_{\sfp}$, are  jointly monic, then $(-)^{\aa}$ is a submonad of $\mathbb{S}$ via the monad morphism $m$.
\end{lemma}

 \begin{proof} We use the notation $\big((-)^{\aa}, \mu, \eta\big)$ and $\big(\mathbb{S}, \mu^S, \eta^S\big)$ for the corresponding monad structures.
 
 (i) Naturality of $m$ is seen from the following diagram  where $a$ ranges over $\ck(A,D)$: 
  $$\xymatrix{X^{\aa}\ar[dd]_{h^{\aa}}\ar[dr]^{(a\ot h)^{\aa}}\ar[rrr]^{m_X}&&&SX\ar[dl]_{\pi_{(a\ot h)}}\ar[dd]^{Sh}\\
  &D^{\aa}\ar[r]_{e}&D&\\
Y^{\aa}\ar[ur]_{a^{\aa}}\ar[rrr]_{m_Y}&&&SY\ar[ul]^{\pi_a}}$$
 The right-hand triangle is the definition of $Sh$.
 
  (ii) Each $m_X$ is monic. This is clear since the cone of all $e\ot a^{\aa}$ is monic.

 (iii) $m$ preserves units. The unit $\eta^S$ of $S$ has components $\eta_X^S:X\to D^{\ck(X,D)}$ defined by 
 $$\pi_a\ot \eta_X^S=a \quad \quad \text{ for all $a:X\to D$.}$$
 Thus, we obtain the following commutative diagram
  $$\xymatrix{&&X^{\aa}
  \ar[d]_{a}\ar[dddrr]^{\eta_X^S}\ar[dddll]_{\eta_X}&&\\
  &&D\ar[dl]_{\eta_D}\ar@{=}[rd]&&\\
  &D^{\aa}\ar[rr]_{e}&&D&\\
X^{\aa}\ar[ur]_{a^{\aa}}\ar[rrrr]_{m_X}&&&&SX\ar[ul]^{\pi_a}}$$
 
 (iv)
  To prove that $m$ preserves multiplication, consider the following diagram:
  $$\xymatrix{X^{\aa\aa}\ar[rr]^{m_{X^{\aa}}}
  \ar[dd]_{\mu_X}
 \ar[rd]^{(\pi_a\ot m_X)^{\aa}}
 &&X^{\aa}\ar[r]^{Sm_X}\ar[d]_{\pi_{\pi_a\ot m_X}}\ar[r]^{Sm_X}&SSX\ar[dd]^{\mu_X^S}\ar[ld]^{\pi_{\pi_a}}\\
  &D^{\aa}\ar[r]_{e}&D&\\
  X^{\aa}\ar[ru]^{a^{\aa}}\ar[rrr]_{m_X}&&&SX\ar[lu]_{\pi_a}}$$
The upper left-hand part and the lower part commute due to the definition of  $m$.  The right-hand upper triangle expresses the definition of $S$ on morphisms,  and the lower one commutes due to Remark \ref{R5.1}.  Therefore, the outside square commutes.
 \end{proof}

\begin{example}\label{E5.7} Let $\ck=\Set$ and $\ca=\Set_{\lambda}$, sets of power less than $\lambda$.

(a) Leinster observed in \cite{L} that the ultrafilter monad is the codensity monad  of $\Set_4\hookrightarrow \Set$ (sets of at most 3 elements). In contrast, $\Set_3\hookrightarrow \Set$ has the codensity monad  defined by
$$TX=\text{ collections of  nonempty subsets of $\cp X$ including either $Y$ or $\overline{Y}$ for every $Y\subseteq X$.}$$

(b) For every infinite cardinal $\lambda$ let $\mathbb{U_{\lambda}}$ be the submonad of the ultrafilter monad $\mathbb{U}$ of all $\lambda$-complete ultrafilters $\mathcal{F}$. Recall that this means that in every disjoint decomposition $e:X\twoheadrightarrow A$ with $|A|<\lambda$  one component lies in $\mathcal{F}$. 

 The codensity monad  of $\Set_{\gl}\hookrightarrow \Set$ is the submonad $\mathbb{U}_{\lambda}$  of the ultrafilter monad $\mathbb{U}$ on all $\gl$-complete ultrafilters, see \cite{ABCPR}. 
\end{example}

\begin{remark}\label{R5.8} Recall that a cardinal $\lambda$ is measurable if there exists a non-principle $\lambda$-complete ultrafilter. $\Set_{\lambda}$ is codense in $\Set$ (i.e., has the trivial codensity monad  $\Id$) iff $\lambda$ is not measurable. This was proved by Isbell in \cite{I}.
\end{remark}

\begin{example}\label{E5.9} Let $\ck=K$-$\Vec$ and $\ca=K$-$\Vec_{\gl}$, spaces of dimension less than $\gl$.

(a) If $\gl$ is an infinite cardinal, then the codensity monad  is analogous to the above example of $\Set_{\gl}\hookrightarrow \Set$, see \cite{ABCPR}. A vector $x$ in $X^{\aa}$ is called \textit{$\gl$-complete} if for every linear decomposition $e:X\to A$ with dim$A<\gl$, we have $e^{\aa}(x)\in\eta_A[A]$. All $\gl$-complete vectors form a submonad of $(-)^{\aa}$. And this is the codensity monad  of $K$-$\Vec_{\gl}\hookrightarrow K$-$\Vec$.

(b) For $\ca$ consisting of $K$ alone the codensity monad  is larger than $(-)^{\aa}$: it assigns to $X$  the space of all \textit{homogeneous} functions from $X^{\ast}$ to $K$ (i.e., those preserving the scalar multiplication). More precisely, $T$ is the subfunctor of $SX=K^{X^{\ast}}$ given by
$$TX=\text{ all homogeneous functions in $K^{X^{\ast}}$.}$$
Indeed, the diagram $C_X$ given by $(X\xrightarrow{a}K)\mapsto K$ has the cone $\pi_a: TX\to K$ formed by restrictions of the projections of $K^{X^{\ast}}$. That is,
$$\pi_a(h) =h(a) \quad \text{ for $h\in TX$, $a\in X^{\ast}$.}$$
To prove that this is a limit cone, let another cone with domain $Z$ be given:
$$\begin{tabular}{cc}
$X\xrightarrow{a}K$\\
\hline \\ [-2.5ex]
$Z\xrightarrow{\ba}K$
\end{tabular}
$$
It is compatible, therefore, for every scalar $\gl\in K$ the morphism  $\gl\ot (-):a\to \gl a$ of $X/\{K\}$ yields 
$$\gl \ot \ba=\overline{\gl \ot a}.$$
Consequently, we can define a function $r: Z\to TX$ by taking $z\in Z$ and putting 
$$r(z):a\mapsto \ba(z) \quad \text{ for $a\in X^{\ast}$.}$$
Then $r(z)$ is homogeneous. This is the desired factorization: $r$ is a linear function with 
$$\pi_a\ot r=\ba \quad \text{ for all $a\in X^{\ast}$.}$$
And it is clearly unique.

(2) In contrast, for $\ca=\{K,\, K^2\}$ in $K$-$\Vec$, the codensity monad  is $(-)^{\aa}$. Indeed, given a cone of $C_X$
$$\begin{tabular}{cc}
$X\xrightarrow{a}K^i$\\
\hline \\ [-2.5ex]
$Z\xrightarrow{\ba}K^i$
\end{tabular}
\qquad \qquad  (i=1,2)
$$
then we again define $r$ by $r(z):a\mapsto \ba(z)$ for $a\in X^{\ast}$. We have to verify that each $r(z)$ is linear, the rest is as above. Homogeneity is verified as before.

To prove additivity,
$$\overline{a_1+a_2}=\overline{a_1}+\overline{a_2} \quad \text{ for $a_1, a_2 \in X^{\ast}$}$$
consider the projections as morphisms  
$$\pi_i:(K^2,\langle a_1,a_2\rangle ) \to (K,a_i) \qquad (i=1,2)$$
of $X/\ca$ which by compatibility yield
$$\pi_i\ot \overline{\langle a_1,a_2\rangle}=\overline{a_i}.$$
That is,
$$\overline{\langle a_1,a_2\rangle }=\langle \overline{a_1}, \overline{a_2}\rangle .$$
We also have a morphism 
$$\pi_1+\pi_2:(K^2,\langle a_1,a_2\rangle)\to (K,a_1+a_2)$$
therefore
$$(\pi_1+\pi_2)\ot \overline{\langle a_1,a_2\rangle}=\overline{a_1+a_2}.$$
Since $(\pi_1+\pi_2)\ot \langle \overline{a_1}, \overline{a_1}\rangle =\overline{a_1}+\overline{a_2}$, the proof is complete.
\end{example}

\begin{example}\label{Top}
Let $\ck=\Top$, the category of topological spaces and continuous maps, and  $\ca=\Top_{\mathrm{f}}$ consist of all finite spaces. The corresponding codensity monad  $\mathbb{T}$ is, as for sets, the ultrafilter monad. More precisely, for every space $X$, $TX$ is the set of all ultrafilters on the underlying set of $X$ with the topology $\tau$ having as a basis all sets of the form
$$\triangle  G=\{ \cu\in TX \mid G\in \cu\}, \quad G \, \text{ open in $X$.}$$

To see this, let $D=\{0,1\}$ be the indiscrete space. This is a cogenerator of $\Top$, and the  space $SX=D^{\Top(X,D)}$ is the indiscrete space $\cp\cp X$ of all collections of subsets of $X$. The proof  that the ultrafilters on the underlying set of a topological space $X$ coincide with $D$-ultrafilters on $X$ is completely analogous to that of Lemma \ref{L3.3}. 

 
To verify that $\tau$  is the topology of $TX$, we just need to show that $\tau$ makes all the morphisms $q(a)$ (see diagram \eqref{eqA0} of Theorem \ref{T5.4}) continuous and jointly initial. That is, $\tau$ is the coarser topology making all $q(a)$ continuous. Indeed,  the open sets of  $A_0$ are of the form 
 $$\hat{H}=\{\cu\in SX \mid  a^{-1}(H)\in \cu\} \qquad \text{for $H$ an open set of $A$,}$$
and $(q(a))^{-1}(\hat{H})=\triangle a^{-1}(H)$. The initiality follows immediately, since, for every open set $G$ of $X$, $\triangle G=\triangle  \chi_G^{-1}(\{1\})$ for $\chi_G$ the characteristic function into the Sierpinski space. 
\end{example}

\begin{example}\label{Top0}
Let $\ck=\Top_0$, the category of $T_0$-topological spaces and continuous maps, and  $\ca$ consist of the finite spaces. The corresponding codensity monad  is the monad of prime open filters. More precisely, for every space $X$, $TX$ is the set of all prime filters on the poset $\Omega X$ of open sets   with the topology having as a basis all sets of the form
$$\square  G=\{ \cu\in TX \mid G\in \cu\}, \quad G \, \text{ open in $X$.}$$
The proof is analogous to the one for posets, using as cogenerator the Sierpinski space.
\end{example}


\begin{thebibliography}{xx}

 \bibitem{ABCPR} J. Ad\'amek, A. Brooke-Taylor, T. Campion, L. Positselski and J. Rosick\'{y},  Colimit-dense subcategories, \textit{Comment. Math. Univ. Carolinae} 60 (2019), 447--462.

\bibitem{BN} B.~Banaschewski and E.~Nelson, Tensor products and bimorphisms, \textit{Canad. Math. Bull.} 19 (1976), 385--402.

\bibitem{Borceux} F. Borceux,  \textit{Handbook of categorical algebra 2,  Categories and structures}. Encyclopedia of Mathematics and its Applications, 51. Cambridge University Press, 1994.

\bibitem{D} B.-P. ~Devlin, \textit{Codensity, Compactness and Ultrafilters} PhD Thesis, University of Edinburgh, 2015.

\bibitem{GH} F.~Galvin and A.~Horn, Operations preserving all equivalence relations, \textit{Proc. AMS} 24 (1970), 521--523.

\bibitem{GU} P. Gabriel and F. Ulmer, \textit{Lokal pr\"asentierbare Kategorien}, Lect. Notes in Math. 221, Springer-Verlag, Berlin, 1971.

\bibitem{I} J.R. Isbell, Adequate subcategories, \textit{Illinois J. Math.} 4 (1960), 541--552.

\bibitem{KG} J.~F.~Kennison and D.~Gildenhuys, Equational completion, model induced triples and pro-objects, \textit{J. Pure Appl. Algebra} 1 (1971), 317--346.

\bibitem{Kelly}G. M. Kelly, \textit{Basic Concepts of
Enriched Category Theory},
Reprints in Theory and Applications of Categories, No. 10, 2005.

\bibitem{K} A.~Kock, Strong functors and monoidal monads, \textit{Archiv der Mathematik} 23 (1972), 113--120.

\bibitem{L} T.~Leinster,  Codensity and the ultrafilter monad, \textit{Theory App. Categories} 28 (2013), 332--370.

\bibitem{Li} F.~E.~J.~Linton, Autonomous equational categories, \textit{J. Math. Mech.} 15 (1966), 637--642.


 \end{thebibliography}
\end{document}